\documentclass[11pt]{article}
\usepackage{amsmath,amssymb,amsthm,amsfonts,amstext,amsbsy,amscd}
\RequirePackage[colorlinks=true,citecolor=blue,urlcolor=blue]{hyperref}

\usepackage{xcolor}
\usepackage{comment}
\usepackage{graphicx}
\usepackage{mathrsfs}

\allowdisplaybreaks[4]
\numberwithin{equation}{section}

\providecolor{darkblue}{rgb}{0.1,0.1,0.6}
\providecolor{darkgreen}{rgb}{0.0,0.4,0.0}

\swapnumbers

\newtheorem{satz}{Satz}[section]

\newtheorem{theorem}[satz]{Theorem}
\newtheorem{proposition}[satz]{Proposition}
\newtheorem{corollary}[satz]{Corollary}
\newtheorem{lemma}[satz]{Lemma}
\newtheorem{assumption}[satz]{Assumption}

\newtheorem{remark}[satz]{Remark}

\newtheorem{example}[satz]{Example}

\DeclareMathOperator{\E}{{\mathbb E}}

\DeclareMathOperator{\R}{{\mathbb R}}

\DeclareMathOperator{\N}{{\mathbb N}}

\DeclareMathOperator{\Var}{Var} \DeclareMathOperator{\Cov}{Cov}

\providecommand{\eps}{\varepsilon}
\renewcommand{\phi}{\varphi}
\renewcommand{\theta}{\vartheta}
\renewcommand{\subset}{\subseteq}

\providecommand{\abs}[1]{\lvert #1 \rvert}
\providecommand{\norm}[1]{\lVert #1 \rVert}

\providecommand{\babs}[1]{{\Bigl\lvert #1 \Bigr\rvert}}
\providecommand{\scapro}[2]{\langle #1,#2 \rangle}
\providecommand{\floor}[1]{\lfloor #1 \rfloor}

\newcommand{\mtc}{\mathcal}
\newcommand{\mbf}{\mathbf}

\newcommand{\mbr}{\mathbb{R}}

\newcommand{\wt}[1]{{\widetilde{#1}}}
\newcommand{\wh}[1]{{\widehat{#1}}}

\newcommand{\ind}[1]{{\mbf{1}(#1)}}
\newcommand{\paren}[1]{\left(#1\right)}

\newcommand{\set}[1]{\left\{#1\right\}}

\newcommand{\LeftEqNo}{\let\veqno\@@leqno}

\newcounter{nbdrafts}
\makeatletter
\newcommand{\checknbdrafts}{
\ifnum \thenbdrafts > 0
\@latex@warning@no@line{**********************************************************************}
\@latex@warning@no@line{* The document contains \thenbdrafts \space draft note(s)}
\@latex@warning@no@line{**********************************************************************}
\fi}

\makeatother

\newtheorem{myassmpt}{}

\newcommand{\setmyassmpttag}[1]{%
  \renewcommand{\themyassmpt}{#1}
}

\begin{document}
\title{Optimal adaptation for early stopping\\ in statistical inverse problems\footnote{Financial support by the DFG via Research Unit 1735 {\it Structural Inference in Statistics} and SFB 1294 {\it Data Assimilation} is gratefully acknowledged.}}
\author{\parbox[t]{3.5cm}{\centering Gilles Blanchard\\
[2mm]
 \footnotesize{\it
Institute of Mathematics\\
Universit\"at Potsdam}\\
{  gilles.blanchard@math.\\uni-potsdam.de}}
\parbox[t]{4.5cm}{\centering Marc Hoffmann
\\[2mm]
 \footnotesize{\it
CEREMADE\\
Universit\'e Paris-Dauphine}\\
\mbox{}hoffmann@ceremade.dauphine.fr
} \parbox[t]{4.5cm}{\centering Markus
 Rei\ss\\[2mm]
 \footnotesize{\it
Institute of Mathematics\\
Humboldt-Universit\"at zu Berlin}\\
\mbox{} mreiss@math.hu-berlin.de}}


\maketitle

\begin{abstract}
For linear inverse problems $Y=\mathsf{A}\mu+\xi$, it is classical to recover the unknown signal $\mu$ by iterative regularisation methods $(\widehat \mu^{(m)}, m=0,1,\ldots)$ and halt at a data-dependent iteration $\tau$
using some stopping rule, typically  based on a discrepancy principle, so that the weak (or prediction) squared-error $\|\mathsf{A}(\widehat \mu^{(\tau)}-\mu)\|^2$ is controlled. In the context of statistical estimation with stochastic noise $\xi$, we study oracle adaptation (that is, compared to the best possible stopping iteration) in strong squared-error $\E\big[\|\widehat \mu^{(\tau)}-\mu\|^2\big]$.

For a residual-based stopping rule oracle adaptation bounds are established for general spectral regularisation methods. The proofs use bias and variance transfer techniques from weak prediction error to strong $L^2$-error, as well as convexity arguments and concentration bounds for the stochastic part. Adaptive early stopping for the Landweber method is studied in further detail and illustrated numerically.
\end{abstract}

\noindent {\it Key words and Phrases:} Linear inverse problems. Early stopping. Discrepancy principle. Adaptive estimation. Oracle inequality. Landweber Iteration.\\
\noindent {\it AMS subject classification:} 65J20, 62G05.

\section{Introduction and main results}

\subsection{Motivation}

\subsubsection*{Statistical linear inverse problems} We wish to recover a signal (a function, an image) from noisy data when the observation of the signal is further challenged by the action of a linear operator.  As an illustrative example, we consider the model of inverse regression in dimension $d=1$ over $[0,1]$. We observe
\begin{equation} \label{eq:inverse regression}
Y_k = A\mu(k/n)+\sigma \xi_k,\;\;k=1,\ldots, n
\end{equation}
where $\mu \in L^2([0,1])$ is the signal of interest, $A:L^2([0,1]) \rightarrow L^2([0,1])$ is a bounded linear operator (with $A\mu$ a continuous function), $\sigma>0$ is a measurement noise level and $\xi_1, \ldots, \xi_n$ are independent standard normal random variables.
An idealised version of \eqref{eq:inverse regression} is given by  the continuous observation of
\begin{equation} \label{eq:WNM}
Y(t)=A\mu(t)+\delta \dot W(t),\;\;t \in [0,1],
\end{equation}
where $\dot W$ is a Gaussian white noise in $L^2([0,1])$ with noise level
\begin{equation} \label{noise level calib}
\delta = \frac{\sigma}{\sqrt{n}}.
\end{equation}
For the asymptotics $n\to\infty$  the rigorous statistical equivalence between \eqref{eq:inverse regression} and \eqref{eq:WNM} goes back to Brown and Low \cite{BL} and was extended to higher dimensions and possibly  $\sigma\to 0$ in Rei\ss\ \cite{Re}. This setting of statistical inverse problems is classical and has numerous practical applications, see  among many other references 
Mair and Ruymgaart \cite{MR},
Cohen {\it et al.} \cite{CHR},
Bissantz {\it et al.} \cite{BHMR} and the  survey by Cavalier \cite{C}.

\subsubsection*{Early stopping and regularization}
Most implemented estimation or recovery methods for $\mu$ are based on a combination of discretisation and iterative inversion or regularisation. Start with an approximation space $V_D \subset L^2([0,1])$ with $\text{dim}(V_D)=D\le n$. First, suppose that \eqref{eq:inverse regression} is observed without noise, {\it i.e.}, $\sigma=0$. An approximation $\mu_D$ for $\mu$ is then obtained by minimising the criterion
$$\|Y-A\mu\|_n^2=\frac1n\sum_{k=1}^n\big(Y_k - A\mu(k/n)\big)^2 \rightarrow \min_{\mu \in V_D}!$$
Using gradient descent (also called Landweber iteration in this context), we obtain the fixed point iteration  for $ \mathsf{A} = A|_{V_D}$:
\begin{equation} \label{algo no noise}
\mu^{(0)}=0,\;\;\mu^{(m+1)}=\mu^{(m)}+ \mathsf{A}^\ast\big(Y-\mathsf{A}\mu^{(m)}\big).
\end{equation}
If $\| \mathsf{A}^\ast\mathsf{A} \|<2$, we have the convergence $\mu^{(m)}\rightarrow \mu_D$ as $m\rightarrow \infty$.

The same program applies when the data are noisy: we fix a large approximation space $V_D$ and transfer our data into the approximating linear model
\begin{equation} \label{eq:linear model}
Y=\mathsf{A}\mu +\sigma \xi
\end{equation}
with  $\mu \in \R^D$, $\mathsf{A} \in \R^{n\times D}$ and $Y,\xi \in \R^n$, with obvious matrix-vector notation. In formal analogy with \eqref{algo no noise} we obtain a sequence of iterations
\begin{equation} \label{algo noise}
\widehat \mu^{(0)}=0,\;\;\widehat \mu^{(m+1)}=\widehat \mu^{(m)}+ \mathsf{A}^\ast\big(Y-\mathsf{A}\widehat \mu^{(m)}\big).
\end{equation}
The presence of a noise term generates a classical conflict as $m$ grows: the iterates $\widehat \mu^{(0)}, \widehat \mu^{(1)}, \ldots, \widehat \mu^{(m)},\ldots$ are ordered with decreasing bias
$\norm{\E[\widehat \mu^{(m)}] - \mu}$ and increasing variance $\E[\norm{\widehat \mu^{(m)} - \E[\widehat \mu^{(m)}]}^2]$, where $\E$ denotes expectation. Thus, early stopping at some iteration $m$ serves as a regularisation method which simultaneously reduces numerical and statistical complexity at the cost of a bias term.

More generally, spectral regularisation methods for linear inverse problems take the
form
$\widehat \mu_{(\alpha)} = f_\alpha(\mathsf{A}^\ast \mathsf{A}) \mathsf{A}^\ast Y$,
where $(f_\alpha)_{\alpha >0}$ is a one-parameter
family of functions $\mathbb{R}_+ \rightarrow \mathbb{R}_+$  satisfying certain generic conditions, see Engl {\it et al.} \cite{EHN}.
In this parametrisation the limit $\alpha\rightarrow 0$ corresponds to less regularisation, with $f_0(u)=1/u$ being the unregularised inverse. In this paper, we will consider such general regularisation schemes using the reparametrisation
$g(t,\lambda) = \lambda^2 f_{t^{-2}}(\lambda^2)$; hence,  a larger  regularisation
parameter $t>0$, to be interpreted as computational time, indicates less regularisation (or a larger variance) and $\lim_{t\to\infty}g(t,\lambda)=1$.

Regularisation methods can be iterative or not, but it is common that the regularisation parameter $t$
is chosen from a fixed in advance grid $t_0 < t_1 < \ldots$, $t_m \rightarrow \infty$ as $m \rightarrow \infty$,
which we still refer to as ``iterations''. Conversely,
any regularisation method defined only on a discrete grid, such as Landweber iteration, can be extended to a continuous regularisation method by appropriate  interpolation, see also Remark \ref{rk:discretisation error} on the discretisation error below.
We note that even for methods that are not intrinsically iterative (for instance Tikhonov regularisation
$g(t,\lambda) = (1 + (t\lambda)^{-2})^{-1}$), computing the estimates for larger $t$ is generally more
resource-intensive since this computation is often numerically less stable.

\subsection{Adaptivity, oracle approach and early stopping}

There are several ways to choose $t = \widehat t  = \widehat t(Y)$ from the data  $Y$ to try to achieve
close to optimal performance.
%
Recent results are formulated within the  oracle approach, comparing the error of $\widehat \mu^{(\widehat t)}$ to the minimal error among $(\widehat\mu^{(t)})_{t>0}$. An idealized oracle
inequality would  take the form
\begin{equation}
  \label{eq:oracletype}
  \forall \mu: \qquad
\E\big[\norm{\mu - \widehat \mu^{(\wh t)}}^2 \big]\lesssim \inf_{t>0} \E\big[\norm{\mu - \widehat \mu^{(t)}}^2\big],
\end{equation}
where  $\lesssim$ indicates inequality up to a constant. Such an inequality can be seen as a convenient way
to transfer an {\em a priori} optimal parameter choice to an {\em a posteriori} one. Namely, given a family of signal  classes $\mathcal{S}_r$ indexed by some regularity
parameter $r$, assume that for any $r$, there exists an {\em a priori} regularisation parameter
choice $t^*_r$ such that $\widehat \mu^{(t^*_r)}$ has minimax-optimal convergence rate over
$\mu \in \mathcal{S}_r$. Then if \eqref{eq:oracletype} holds for the data-dependent
choice $\wh{t}$, it implies that this {\em a posteriori} rule also enjoys minimax convergence rates
over the regularity classes $(\mathcal{S}_r)$.
For more concrete statements, such as adaptation over classes of Sobolev ellipsoids, see {\it e.g.} Cavalier  \cite{C}.
The advantage of an oracle inequality is that it is a stronger
statement than a posteriori optimality over a certain family: it implies adaptation for any given individual
signal $\mu$.

 Typical statistical methods to determine $\wh t$ use (generalized) cross validation (Wahba \cite{Wa}), penalized empirical risk minimisation (Cavalier and Golubev \cite{CG}) or Lepski's balancing principle for inverse problems (Math\'e and Pereverzev \cite{MP}). Some of these rules can be transferred to deterministic inverse problems see {\it e.g.} Pereverzev and Schock \cite{PS}. All these methods lead to
some form of oracle inequality of the type \eqref{eq:oracletype}, and thus to
minimax adaptation over suitable regularity classes.
They share, however, the drawback that  the estimators $\widehat \mu^{(t)}$ have first to be computed up to some maximal iteration (or general parameter choice) $T$, prescribed prior to data analysis, and then be compared to each other in some way in order to finally determine $\wh{t}$. Computing all estimators in the first place seems an undesirable waste of resources,
in particular in hindsight those for $t> \widehat t$.

This state of the art for statistical inverse problems stands in contrast with the  deterministic inverse problem setting, in which the noise
$\xi$ in the model \eqref{eq:linear model} is assumed fixed and of norm bounded by~1.
In that setting, 
it is well-known that the {\em discrepancy principle},
consisting in stopping for the first iteration such that the residual
$R_t=\|Y-\mathsf{A}\widehat \mu^{(t)}\|$
is smaller than $\tau \sigma$ (for some constant $\tau>1$), is an {\em a posteriori}
rule enjoying, under mild conditions, optimal adaptivity in the deterministic sense over signal classes $\mathcal{S}_r$ defined
by source conditions (see  Engl {\it et al.} \cite{EHN} for a precise analysis).  Yet, there is no easy transfer of results from the deterministic setting to the statistical; let us point out in particular
that (a) the optimal rates in the statistical and deterministic settings are different,
(b) under the white noise model the typical order of the squared norm of the noise
is not constant but grows linearly with the output space dimension, and (c) in the statistical setting,
cancellation and law of large number effects due to independence of the noise coordinates play a crucial role,
whereas the deterministic setting is worst-case (or ``adversarial'') under the constraint $\norm{\xi}\leq 1$.

  The stopping rules we  consider will in fact be similar to the discrepancy principle.
  Hansen \cite{Ha} discusses practical issues, in particular
  modifications of the discrepancy principle for statistical noise in Chapter 5. For statistical inverse problems  Blanchard and Math\'e \cite{BM}, Lu and Math\'e \cite{LuM} introduce regularised residuals in order to encompass the fact that $R_t^2$ becomes arbitrarily large as $D$ grows.
  Our approach will not require such further regularisations.

It turns out that one cannot establish  full oracle adaptation of the form
\eqref{eq:oracletype}; in a nutshell, the proposed stopping rule will
only be oracle adaptive with respect to a range of possible
regularisation parameters $t$ depending on the the variance level of
estimator $\widehat \mu^{(t)}$ and on the total discretisation dimension
$D$. In a previous paper \cite{ourpapercutoff}, it has been shown for
the specific case of truncated SVD regularisation that a lower
bound holds which prevents adaptation over the full range for early stopping rules. If the oracle error is too small, then any stopping rule must incur an additional error of larger order. In the present
paper on general regularisation schemes this means that we cannot start at $t=0$, but can only consider stopping after a certain minimal number $t_0>0$ of iterations.

Let us point out that in the deterministic approach  the noise level $\sigma$ must be known in advance in order to apply successfully a discrepancy principle, an observation going back to Bakushinski \cite{B}. 
An advantage of the statistical approach of \eqref{eq:inverse regression} and \eqref{eq:linear model} is that the noise level $\sigma^2$ can be estimated from the data $Y$, see {\it e.g.} Golubev \cite{G}. This is transparent in the limiting model \eqref{eq:WNM} since $\delta^2$ related to $\sigma^2$ and the number $n$ of observations in \eqref{noise level calib} is identified by the continuous observation of $(Y(t), t \in [0,1])$ thanks to its quadratic variation.

\subsection{Mathematical setting}

Our analysis for the model \eqref{eq:linear model} will use the representation of estimators in the singular value decomposition (SVD):
 let $(\mathsf{A}^\ast \mathsf{A})^{1/2}$ have eigenvalues
$$1 \ge  \lambda_1\ge \lambda_2 \ge \ldots \ge \lambda_D >0$$
with a corresponding orthonormal basis of eigenvectors $(v_1,\ldots, v_D)$ in terms of the empirical scalar product $\scapro{a}{b}_n=\frac1n\sum_{i=1}^Da_ib_i$, $a,b\in\R^D$. We obtain the diagonal SVD model in terms of $\mu_i=\langle \mu,v_i \rangle_n$, $Y_i=\langle Y,w_i\rangle_n, w_i=\tfrac{\mathsf{A}^\ast v_i}{\|\mathsf{A}v_i\|_n}$ and
\begin{equation} \label{eq: def seq model}
Y_i = \lambda_i \mu_i+ \delta \varepsilon_i,\;\;i=1,\ldots, D,
\end{equation}
where the $\varepsilon_i$ are independent standard Gaussian random variables and $\delta=\frac{\sigma}{\sqrt n}$ is the noise level. Our objective is to recover the signal $\mu = (\mu_i)_{1 \le i \le D}$ with best possible accuracy from the data $(Y_i)_{1 \le i \le D}$.

We do not rely on the calculation of the  SVD, which  is computationally rarely feasible, but for the analysis we employ the following SVD representation of a linear estimator $\widehat \mu^{(t)}$:
\begin{equation}\label{Eqmufilter}
\widehat \mu_i^{(t)} = \gamma_i^{(t)}\lambda_i^{-1}Y_i= \gamma_i^{(t)}\mu_i+ \gamma_i^{(t)}\lambda_i^{-1}\delta\varepsilon_i,\;\;i=1,\ldots D.
\end{equation}
We thus specify linear estimation procedures by  {\it filters} $(\gamma_i^{(t)})_{i=1,\ldots,D,t \ge 0}$ that satisfy
$\gamma_i^{(t)} \in [0,1]$, $\gamma_i^{(0)}=0$ and $\gamma_i^{(t)} \uparrow 1$ as $t\rightarrow \infty$.
These filter properties are satisfied by typical spectral regularisation methods, see Example \ref{detailed filters} below.

The squared bias-variance decomposition of the mean integrated squared error
(MISE) writes
$$\E\big[\|\widehat \mu^{(t)}-\mu\|^2\big] = B_t^2(\mu)+V_t$$
with
\begin{equation} \label{def biais et var forts}
B_t^2(\mu)  = \sum_{i = 1}^D(1-\gamma_i^{(t)})^2\mu_i^2\;\;\text{and}\;\;
V_t  =\delta^2\sum_{i = 1}^D(\gamma_i^{(t)})^2\lambda_i^{-2}.
\end{equation}
In distinction with the weak norm quantities defined below, we shall call $B_t(\mu)$ {\it strong bias} and $V_t$ {\it strong variance}.



The  estimators we consider take the form $\widehat \mu^{(\tau)} = (\widehat \mu_i^{(\tau)})_{1 \le i \le D}$ where
$\widehat \mu_i^{(\tau)} = \gamma_i^{(\tau)}\lambda_i^{-1}Y_i.$
As for the discrepancy principle, we search for a stopping rule $\tau$ based on the information generated by the  residual
\begin{equation} \label{def continuous residual}
R_t^2=\|Y-\mathsf{A}\widehat \mu^{(t)}\|^2=\sum_{i = 1}^D(1-\gamma_i^{(t)})^2Y_i^2,\quad t\ge 0.
\end{equation}
The information provided by $R_t^2$ becomes transparent by considering  the {\it weak or prediction norm $\norm{v}_{\mathsf A}=\norm{{\mathsf A}v}$} and by decomposing  the weak norm error $ \E[\|\widehat \mu^{(t)}-\mu\|_{\mathsf{A}}^2] = B_{t,\lambda}^2(\mu)+V_{t,\lambda}$ into a {\it weak squared bias} $B_{t,\lambda}^2(\mu)$ and a {\it weak variance} $V_{t,\lambda}$:
\begin{align}
 B_{t,\lambda}^2(\mu)&=\|\E[\widehat \mu^{(t)}]-\mu\|_{\mathsf{A}}^2=\sum_{i = 1}^D\big(1-\gamma_i^{(t)}\big)^2\lambda_i^2\mu_i^2,\label{DefWeakBias}\\
 V_{t,\lambda}& = \E\Big[\|\widehat \mu^{(t)}-\E[\widehat \mu^{(t)}]\|_{\mathsf{A}}^2\Big]=\delta^2\sum_{i = 1}^D\big(\gamma_i^{(t)}\big)^2.\label{DefWeakVar}
 \end{align}
Then a bias-corrected residual $R_t^2$  estimates the weak squared bias:
\begin{equation} \label{WeakBiasEst}
\E\Big[R_t^2-\delta^2\sum_{i = 1}^D(1-\gamma_i^{(t)})^2\Big]=B_{t,\lambda}^2(\mu).
\end{equation}
We are led to consider stopping rules of the form
\begin{equation} \label{def early stopping}
\tau = \inf\big\{t\ge t_0:R_t^2\le \kappa\big\}
\end{equation}
for some initial smoothing step $t_0\ge 0$ and a threshold value $\kappa>0$. A residual $R_t^2$
larger than an appropriate choice of $\kappa$ indicates strong evidence that
there is relevant information about $\mu$ beyond $t$.

\subsection{Overview of main results} \label{sec:main results}

In Section \ref{sec:upper bounds}, we first establish in Proposition \ref{PropMainBound} an oracle inequality for $\widehat \mu^{(\tau)}$ in weak-norm by comparing to $\widehat \mu^{(t^\ast)}$, where
the deterministic stopping index
\begin{equation}\label{EqOracleProxy}
t^\ast = t^\ast(\mu) = \inf\big\{t\ge t_0: \E_\mu[R_t^2] \le \kappa\big\}
\end{equation}
is interpreted as an {\it oracle proxy} and $\E_\mu$ emphasizes the dependence of the expectation on $\mu$. We transfer the weak estimates into strong estimates thanks to appropriate assumptions on the filter functions $(\gamma_i^{(t)})_{1 \le i \le D}$ as well as on ${\mathsf{A}}$, {\it i.e.}, on  the singular values $(\lambda_i)_{1 \le i \le D}$. Under the mild condition $\delta^2\sqrt D\lesssim B_{t^\ast,\lambda}^2(\mu)\lesssim \delta^2D$  we then establish in Theorem \ref{ThmOrIneqStrong} and its Corollary \ref{CorStrongOrIneq} the oracle-type strong norm inequality
\begin{equation} \label{oracle like strong}
\E\big[\norm{\widehat\mu^{(\tau)}-\mu}^2\big]\lesssim \E\big[\norm{\widehat\mu^{(t^\ast)}-\mu}^2\big].
\end{equation}
Let us emphasize  that classical interpolation arguments between Hilbert scales, usually applied to control the approximation error under the discrepancy principle ({\it e.g.}, Section 4.3 in Engl {\it et al.} \cite{EHN}), cannot be used for an oracle and thus non-minimax approach. It remains to investigate the performance of the deterministic oracle proxy $t^\ast$  in connection with the choice of the threshold $\kappa$ that is free in the above estimates. This is the topic of Section \ref{sec:choice of kappa}. We argue that the choice $\kappa=D\delta^2$, up to deviations of order $\sqrt D\delta^2$, {\it e.g.} due to variance estimation, yields rate-optimal results. The study is conducted by comparing $t^\ast$ with
{\it weakly} and {\it strongly balanced oracles}
\begin{align}
t_{\mathfrak w} &= t_{\mathfrak w}(\mu)=\inf \big\{t \ge t_0: B_{t,\lambda}^2(\mu)\le V_{t,\lambda}\big\},\label{Eqtw}\\
t_{\mathfrak s} &= t_{\mathfrak s}(\mu)=\inf \big\{t \ge t_0: B_t^2(\mu)\le V_t\big\}.\label{Eqts}
\end{align}
We obtain the bias-variance balance $B_{t_{\mathfrak w},\lambda}^2(\mu)=V_{t_{\mathfrak w},\lambda}$ unless $ B_{t_0,\lambda}^2(\mu)< V_{t_0,\lambda}$ and by the monotonicity of weak bias and variance in $t$ it easily follows that
\begin{equation}\label{EqBalNarOracle}
\E\big[\|\widehat \mu^{(t_{\mathfrak w})}- \mu\|_{\mathsf{A}}^2\big] \le 2\inf_{t\ge t_0} \E\big[\|\widehat \mu^{(t)}- \mu\|_{\mathsf{A}}^2\big]
\end{equation}
and analogously for the strongly balanced oracle in strong norm. In a concrete setting the residual, bias and variance as a function of $t$ together with the indices $\tau$, $t^\ast$, $t_{\mathfrak w}$ and $t_{\mathfrak s}$ are visualized by Figure \ref{Fig2} (right) in Section \ref{discussion and numerics} below. The balanced oracles take over the role of the classical oracles (as error minimisers) and form natural benchmarks for sequential stopping rules.
We establish in Theorem \ref{thm: oracle end}
a bound in terms of the rescaled weak oracle error such that, in general,
\[\E\big[\norm{\widehat\mu^{(\tau)}-\mu}^2\big]\lesssim\max\Big(\E\big[\norm{\widehat\mu^{(t_{\mathfrak s})}-\mu}^2\big],
t_{\mathfrak w}^2\E\big[\norm{\widehat\mu^{(t_{\mathfrak w})}-\mu}_{\mathsf A}^2\big]\Big).
\]
For a polynomial decay of singular values this bound becomes a more tractable interpolation-type bound (Corollary \ref{CorTauws}). Assuming that the spectral method has a sufficiently large qualification, Theorem \ref{thm: oracle end} implies in particular that $\widehat\mu^{(\tau)}$ attains the optimal rate $\delta^{2\beta/(2\beta+2p+d)}$ for matrices $\mathsf A$ with a degree $p$ of ill-posedness  and signals $\mu$ in Sobolev-type ellipsoids of dimension $d$ with a regularity $\beta$ ranging within the appropriate adaptation interval (Corollary \ref{CorTauMinimax} and \cite{ourpapercutoff}).

Section \ref{discussion and numerics} studies more specifically the early stopping rule for  Landweber iterations. First, a  large class of signals $\mu$ is identified for which an oracle inequality
\begin{equation*}
\E\big[\norm{\widehat\mu^{(\tau)}-\mu}^2\big]\le C_{\tau,\mathfrak s}\E\big[\norm{\widehat\mu^{(t_{\mathfrak s})}-\mu}^2\big]
\end{equation*}
holds (Corollary \ref{CorLWOrineq}). Then some numerical results show the scope and the limitations of adaptive early stopping,  confirming the theoretical findings.

\section{Oracle proxy bounds} \label{sec:upper bounds}

We consider the family of linear estimators $(\widehat\mu{(t)})_{t\ge 0}$ from \eqref{Eqmufilter} characterised by the filters $(\gamma_i^{(t)})_{i=1,\ldots,D,t\ge 0}$.
Recall the basic assumptions:
$t \mapsto \gamma_i^{(t)}$ is a nondecreasing continuous function with $\gamma_i^{(0)}=0$, and $\gamma_i^{(t)} \uparrow 1$ as $t\rightarrow \infty$.

Given  the residual
$R^2_t$ from \eqref{def continuous residual},
we introduce the residual-based stopping rule
$\tau = \inf\big\{t\ge t_0:R_t^2\le \kappa\big\}$ from \eqref{def early stopping}.
Since $R_{t}^2\downarrow 0$ holds for $t\uparrow\infty$ and $t\mapsto R_{t}^2$ is continuous, we  have $R_{\tau}^2=\kappa$ unless $R_{t_0}^2<\kappa$ already, in which case $\tau=t_0$. Also, recall the oracle proxy $t^\ast=\inf\big\{t\ge t_0\,:\,\E[R_{t}^2]\le\kappa\big\}$, which by the same argument satisfies $\E[R_{t^\ast}^2]=\kappa$, unless already $\E[R_{t_0}^2]<\kappa$ holds, implying $t^\ast=t_0$.

\subsection{Upper bounds in weak norm} \label{upper in weak norm}

\begin{proposition}\label{PropMainBound}
The following inequality holds in weak norm:
\begin{equation} \label{eq:propmainbound}
\E\big[\norm{\widehat\mu^{(\tau)}-\widehat\mu^{(t^\ast)}}_{\mathsf{A}}^2\big] \le \Big(2D\delta^4+4\delta^2 B_{t^\ast,\lambda}^2(\mu)\Big)^{1/2}.
\end{equation}
This implies the oracle-type inequality
\begin{equation} \label{oracleexact}
\E\big[\norm{\widehat\mu^{(\tau)}-\mu}^2_{\mathsf{A}}\big]^{\frac{1}{2}} \le  \E\big[\norm{\widehat\mu^{(t^\ast)}-\mu}^2_{\mathsf{A}}\big]^{\frac{1}{2}} +\Big(2D\delta^4+4\delta^2 B_{t^\ast,\lambda}^2(\mu)\Big)^{1/4} .
\end{equation}
\end{proposition}

\begin{proof}
The main (completely deterministic) argument uses consecutively the definition of the weak norm, the inequality $(A-B)^2\le \abs{A^2-B^2}$ for $A,B\ge 0$ and the bounds $R_\tau^2= \kappa\ge \E[R_{t^\ast}^2]$ for $\tau>t^\ast\ge t_0$ and $R_\tau^2\le \kappa= \E[R_{t^\ast}^2]$ for $t^\ast>\tau\ge t_0$:
\begin{align*}
\norm{\widehat\mu^{(t^\ast)}-\widehat\mu^{(\tau)}}_{\mathsf{A}}^2
&=
\sum_{i=1}^D (\gamma_i^{(t^\ast)}-\gamma_i^{(\tau)})^2Y_i^2\\
&\le \sum_{i=1}^D \abs{(1-\gamma_i^{(t^\ast)})^2-(1-\gamma_i^{(\tau)})^2}Y_i^2\\
&=(R_{t^\ast}^2-R_\tau^2){\bf 1}(\tau>t^\ast)+(R_{\tau}^2-R_{t^\ast}^2){\bf 1}(\tau<t^\ast)\\
&\le (R_{t^\ast}^2-\E[R_{t^\ast}^2]){\bf 1}(\tau>t^\ast)+(\E[R_{t^\ast}^2]-R_{t^\ast}^2){\bf 1}(\tau<t^\ast)\\
&\le\abs{R_{t^\ast}^2-\E[R_{t^\ast}^2]}\\
&= \babs{\sum_{i=1}^D(1-\gamma_i^{(t^\ast)})^2\big(\delta^2(\eps_i^2-1)+2\lambda_i\mu_i\delta\eps_i\big)}.
\end{align*}
Note that  the passage from the second to the third line simply follows from the uniform monotonicity of filters, {\it i.e.},  $\gamma_i^{(t^\ast)}\le \gamma_i^{(\tau)}$  for all $i$ if $t^\ast\le\tau$ and $\gamma_i^{(t^\ast)}\ge \gamma_i^{(\tau)}$ for all $i$ if $t^\ast\ge\tau$.
By bounding the variance of the first term (applying the Cauchy-Schwarz inequality) and using $\abs{1-\gamma_i^{(t)}}\le 1$, $\Var(\eps_i^2)=2$ and $\Cov(\eps_i^2,\eps_i)=0$,
this implies:
\begin{align*}
\E\big[\norm{\widehat\mu^{(t^\ast)}-\widehat\mu^{(\tau)}}_{\mathsf{A}}^2\big]
&\le \Big(2\delta^4\sum_{i=1}^D (1-\gamma_i^{(t^\ast)})^4+ 4\delta^2\sum_{i=1}^D(1-\gamma_i^{(t^\ast)})^4\lambda_i^2\mu_i^2\Big)^{1/2}\\
&\le \Big(2D\delta^4+4\delta^2 B_{t^\ast,\lambda}^2(\mu)\Big)^{1/2}.
\end{align*}
From this first inequality the second follows by the triangle inequality.
\end{proof}

Let us point out that Proposition \ref{PropMainBound} continues to hold under minimal assumptions on the noise: the variables $\eps_i$ need merely match the first four Gaussian moments.

The last term in the right-hand side of \eqref{oracleexact} is of the order of the geometric mean of $B_{t^\ast,\lambda}$ and $\delta$, and thus asymptotically negligible whenever the oracle proxy squared-error is of larger order than $\delta^2$ 
(since $\E[\norm{\widehat\mu^{(t^\ast)}-\mu}^2_{\mathsf{A}}]\ge B^2_{t^\ast,\lambda}$). Consequently,  the oracle-type inequality \eqref{oracleexact} is asymptotically exact, in the sense that

\begin{equation} \label{eq:def exact oracle}
\E\big[\norm{\widehat\mu^{(\tau)}-\mu}_{\mathsf{A}}^2\big] \le \big(1+o(1)\big)\E\big[\norm{\widehat\mu^{(t^\ast)}-\mu}^2_{\mathsf{A}}\big]
\end{equation}
as $\delta \rightarrow 0$, whenever the oracle-type squared-error on the right-hand side
is of larger order than $\sqrt{D}\delta^2$. Our stopping rule thus gives reliable estimators when the weak variance is at least of order $\sqrt{D}\delta^2$, and we henceforth choose the initial smoothing step $t_0$ as
\begin{equation}\label{Eqt0}
 t_\circ=\inf\{t\ge 0\,:\,V_{t,\lambda}=C_\circ D^{1/2}\delta^2\}\quad\text{ for some constant }C_\circ\ge 1.
\end{equation}
Note that $t_\circ$ is well defined for $C_\circ<D^{1/2}$ since $V_{t,\lambda}$ increases from $0$ at $t=0$ to $D\delta^2$ as $t\uparrow \infty$, and that it is easily computable, since $V_{t,\lambda}$ is obtained as squared norm of the estimation method applied to the data $\delta \mathsf{A}\textbf{1}$ with ${\textbf 1}=(1,\ldots,1)^\top$. Moreover, we find back exactly the critical order from the lower bound in \cite{ourpapercutoff}.
\begin{remark}[Controlling the discretisation error] \label{rk:discretisation error}
  For the discrete stopping rule $\widehat m=\inf\{m\in\N\,:\,m\ge t_0,\,R_m^2\le \kappa\}$ we obtain, using $(A+B)^2\le 2(A^2+B^2)$, $(A-B)^2\le\abs{A^2-B^2}$ for $A,B\ge 0$,
  filter monotonicity and $\widehat m \ge \tau$:
\[ \E\big[\norm{\widehat\mu^{(\widehat m)}-\widehat\mu^{(\tau)}}_{\mathsf{A}}^2\big] \le 2\E\Big[B_{\tau,\lambda}^2-B_{\widehat m,\lambda}^2\Big] +2\delta^2\E\Big[\sum_{i=1}^D\big(\gamma_i^{(\widehat m)}-\gamma_i^{(\tau)}\big)^2\eps_i^2\Big].
\]
By $\widehat m-1<\tau$ and the filter monotonicity we  further  bound the right-hand side by
\[  2\max_{m=1,\ldots,D}\Big(B_{m-1,\lambda}^2-B_{m,\lambda}^2 +\delta^2\E\big[{\textstyle\max_{i\le D}\eps_i^2}\big]\norm{\gamma^{(m)}-\gamma^{(m-1)}}_{}^2\Big).
\]
Note that the filter differences are usually not large; for Landweber iteration,  for instance, $\norm{\gamma^{(m)}-\gamma^{(m-1)}}_{}^2\le \sum_{i=1}^D\lambda_i^4$. Because of $\E[\max_{i\le D}\eps_i^2]\lesssim \log D$, the second term is usually of order $\delta^2\log D$ and much smaller  than the error term in the oracle inequality \eqref{oracleexact}. The bias difference term depends on the signal and does not permit a universal bound, but observe that $\widehat m$ stops later than $\tau$ (or at $\tau$) and thus $\widehat\mu^{(\widehat m)}$  incurs less bias in the error bound than $\widehat\mu^{(\tau)}$.
\end{remark}

\subsection{Upper bounds in strong norm} \label{upper in strong norm}

Most common filter functions used in inverse problems are
obtained from {\em spectral regularisation methods} of the form
\begin{equation}
\label{eq:specreg}
\gamma^{(t)}_{i} = g(t,\lambda_i),
\end{equation}
where $g(t,\lambda)$ is a regulariser function
$\mbr_+ \times \mbr_+ \rightarrow [0,1]$, see for instance Engl {\it et al.}, Chapter 4 \cite{EHN} (with the notation
$g(t,\lambda)=\lambda^2 g_{t^{-2}}(\lambda^2)$ in terms of their function $g_\alpha$\,).  Let us collect all required properties and discuss conditions under which they are fulfilled.

A first set of assumptions concerns the regulariser function:

\begin{assumption}[{\bf R}]
\setmyassmpttag{{\bf R1}}
\begin{myassmpt}\label{AssGMon}
  The function $g(t,\lambda)$ is nondecreasing in $t$ and $\lambda$, continuous in $t$ with $g(0,\lambda)=0$
  and $\lim_{t \rightarrow \infty} g(t,\lambda) = 1$ for any fixed $\lambda>0$.
\end{myassmpt}
\setmyassmpttag{{\bf R2}}
\begin{myassmpt}\label{AssGMonRatio}
  For all $t\ge t'\ge t_\circ$, the function $\lambda \mapsto \frac{1-g(t',\lambda)}{1-g(t,\lambda)}$
  is nondecreasing.
\end{myassmpt}
\setmyassmpttag{{\bf R3}}
\begin{myassmpt}\label{AssGComp}
  There exist positive constants $\rho, \beta_-, \beta_+$
  such that for all $t \ge t_\circ$ and $\lambda \in (0,1]$, we have
  \begin{equation} \label{def rho}
   \beta_- \min \big(\paren{t \lambda}^{\rho},1\big) \le g(t,\lambda)  \le  \min \big(\beta_+ \paren{t \lambda}^{\rho},1\big).
    \end{equation}
\end{myassmpt}
\end{assumption}

\ref{AssGMonRatio} is not needed given \ref{AssGComp} if we allow less accurate control in the constants of Lemma \ref{One-sided Karamata relations} (see Proposition \ref{prop:suffcond} and its proof below). Still, it is usually satisfied. The value $\rho$ in \ref{AssGComp} should be distinguished from the qualification of a regularisation method, as introduced in Corollary \ref{CorTauMinimax} below.
While the qualification is intended to control the approximation error, the constant $\rho$ introduced in \eqref{def rho} guarantees instead the control of $\E_0[R_t^2]$  (expectation under null signal)
and $V_t,V_{t,\lambda}$ for large $D$, a pure noise property independent of the signal.

\begin{example} \label{detailed filters} Let us list some commonly used filters ({\it cf.} Engl {\it et al.} \cite{EHN}) that all can be directly seen to satisfy Assumption {\bf (R)} with $\rho=2$ in \ref{AssGComp}.
 \begin{enumerate}
 \item The Landweber filter, as developed in \eqref{algo noise} of the introduction, is obtained by gradient descent of step size 1 (note $\lambda_1 \le 1$):
$$\widehat\mu^{(m)}=\sum_{i=0}^{m-1}  (I-\mathsf{A}^\ast\mathsf{A})^i \mathsf{A}^\ast Y=(I-(I-\mathsf{A}^\ast \mathsf{A})^m) (\mathsf{A}^\ast\mathsf{A})^{-1}\mathsf{A}^\ast Y.$$
   When interpolating with $t=\sqrt{m}$ for $t\geq 1$ this yields $g(t,\lambda)=1-(1-\lambda^2)^{t^2}$;
   we further interpolate by $g(t,\lambda) = (t\lambda)^2$ for $t\in[0,1]$ for technical convenience: then
for assumption {\bf (R3)} $\beta_-=\frac{1}{2},\beta_+=1$ works in \eqref{def rho}
for all $t\geq 0, \lambda \in [0,1]$.
 \item The Tikhonov filter $g(t,\lambda)=(1+(t\lambda)^{-2})^{-1}$ is obtained from the minimisation in $\mu\in\R^D$
 \begin{equation} \label{minimTikho}
   \|Y-\mathsf{A}\mu\|^2+t^{-2}\|\mu\|^2\rightarrow\text{min}_\mu!
\end{equation}
Assumption {\bf (R3)} is satisfied with $\beta_-=\frac{1}{2},\beta_+=1$.
\item The $m$-fold iterated Tikhonov estimator $\widehat\mu^{(\alpha,m)}$ is obtained by minimising iteratively in $m$ the criterion  \eqref{minimTikho}, but with penalty $\alpha^{2}\norm{\mu-\widehat\mu^{(\alpha,m-1)}}^2$, where $\widehat\mu^{(\alpha,1)}$ is the standard Tikhonov estimator with $t=\alpha^{-1}$. The reparametrisation $t=\sqrt m$ yields the filters $g_\alpha(t,\lambda)=1-(1+\alpha^{-2}\lambda^2)^{-t^2}$.
Assumption {\bf (R3)} is satisfied with $\beta_-=\frac{\alpha^{-2}}{2},\beta_+=\alpha^{-2}$.
\item Showalter's method or asymptotic regularisation is the general continuous analogue of iterative linear regularisation schemes. Its  filter is given by $g(t,\lambda)=1-e^{-t^2\lambda^2}$. Assumption {\bf (R3)} is satisfied with $\beta_-=\frac{1}{2},\beta_+=1$.
\end{enumerate}
\end{example}

A second  assumption concerns the spectrum and is satisfied, for instance, for singular values of the form $\lambda_i=ci^{-1/\nu}(\log(i+1))^p$ with some $c,\nu>0$, $p\in\R$.
\begin{assumption}[{\bf S}]
  There exist constants $\nu_-,\nu_+>0$ and $L \in \N$ such that for all
   $L \le k \le D$:
\begin{equation}
  \label{Critbis}
     0< L^{-1/\nu_-} \le \frac{\lambda_{k}}{\lambda_{\lceil \frac{k}{L} \rceil} },
     \;\;\text{ and }\;\;
    \frac{\lambda_{k}}{\lambda_{\lfloor \frac{k}{L} \rfloor} } \le L^{-1/\nu_+} < 1.
\end{equation}
\end{assumption}

%

  The indices $\nu_-$ and $\nu_+$ are related to the so-called lower and upper Matuszewska indices of the function $F(u) = \#\set{i:\lambda_i \ge u}$  in the theory of  $\mathcal O$-regularly varying functions, see Bingham {\it et al.}
  \cite{BGT}, Section 2.1. 
  In classical definitions
  these indices
  are defined asymptotically.
  Since we aim at non-asymptotic results, we  require a version holding
  for all $k$; to account for possible multiple eigenvalues at the beginning of the sequence,
  we allow $L$ to be an arbitrary integer (typically $L$ would be larger than the multiplicity
  of $\lambda_1$).
For connections to inverse and singular value problems in numerical analysis see Djurcic {\it et al.} \cite{DNT} or Fleige \cite{F}.

Assumptions {\bf (S)} and {\bf (R)} on the filter functions $\gamma^{(t)}_i$ and on the spectrum of $\mathsf A$ enable us to develop our theory, e.g. transfer weak norm estimates for the bias and the variance terms into estimates in strong norm, via the following estimates:

\begin{proposition}
  \label{prop:suffcond}
  Suppose Assumptions {\bf (S)} and {\bf (R)} are satisfied
  with $\rho>\nu_+$ 
  and the filter functions given by \eqref{eq:specreg}.
  Then there exist constants $\pi \ge 1,\, C_{V,\lambda}\ge 1,\,c_\lambda>0,\, C_{\ell_1,\ell_2}\ge 1$, depending only on $\rho,\beta_-,\beta_+,L,\nu_-,\nu_+$, such that the following properties are satisfied:
\setmyassmpttag{{\bf A1}}
\begin{myassmpt}
\label{AssBias2}
For all  $t \ge t' \ge t_\circ$, the sequence $\Big(\frac{1-\gamma_{i}^{(t')}}{1-\gamma_{i}^{(t)}}\Big)_{i=1,\ldots,D}$ with values in $[0,\infty]$ is
nonincreasing in $i$.
\end{myassmpt}

\setmyassmpttag{{\bf A2}}
\begin{myassmpt}\label{AssMon}
For all $i'\le i$ and $t\ge  t_\circ$, we have $\gamma_i^{(t)}\le\gamma_{i'}^{(t)}$.
\end{myassmpt}
\setmyassmpttag{{\bf A3}}
\begin{myassmpt}\label{AssVar}
For some $\pi\ge 1$ there exists a constant $C_{V,\lambda}\ge 1$ so that for all $t \ge t' \ge t_\circ$, we have
$$V_{t}\le C_{V,\lambda}(V_{t,\lambda}/V_{t',\lambda})^\pi V_{t'}.$$
\end{myassmpt}

\setmyassmpttag{{\bf A4}}
\begin{myassmpt}\label{AssLambdaL2}

There exists $c_\lambda>0$ such that for every $k=1,\ldots, D$:
$$\frac 1k \sum_{i=1}^k\lambda_i^{-2}\ge c_\lambda^2 \lambda_k^{-2}.$$
\end{myassmpt}

\setmyassmpttag{{\bf A5}}
\begin{myassmpt}\label{Assl1l2}
There exists a constant $C_{\ell^1, \ell^2}$ such that for all $t\ge t_\circ$ we have
$$\sum_{i=1}^D\gamma_i^{(t)}\le C_{\ell^1,\ell^2} \sum_{i=1}^D(\gamma_i^{(t)})^2.$$
\end{myassmpt}

\end{proposition}

 The condition $\rho>\nu_+$ is often encountered in statistical inverse problems,  ensuring, independently of $D$, a control of the variances of the estimators.
The proof of Proposition \ref{prop:suffcond} is delayed until Appendix \ref{proof:prop:suffcond}. In Appendix \ref{proof:LemVart} we also present the proof of the following result, which gives the strong-to-weak variance order $V_t\thicksim (t\wedge \lambda_D^{-1})^2V_{t,\lambda}$ in this framework.

\begin{lemma}\label{LemVart}
  Under Assumptions {\bf (R)} and {\bf (S)}, 
  we have for all $t\ge t_\circ$
\[ V_{t,\lambda}\le C_V(t \wedge \lambda_D^{-1})^{-2} V_t\text{ with }C_V=\tfrac{L^{1+2/\nu_-}}{L-1}\beta_-^{-2}.\]
Under Assumptions {\bf (R)}, {\bf (S)} with $\rho>1+\nu_+/2$ we have for all $t\ge t_\circ$:
\[ V_{t,\lambda}\ge c_V(t\wedge \lambda_D^{-1})^{-2} V_t,\]
with
\[
c_V :=  \min\bigg(1,\bigg(\tfrac{C_\circ \sqrt{D}(1-L^{1-2\rho/\nu_+})}{(L-1)\beta_+}\bigg)^{\frac{1}{\rho}}\bigg)
\tfrac{(1-L^{1-(2\rho-2)/\nu_+})\beta_-^2}{L^{2(\rho+1)/\nu_-}}.
\]
\end{lemma}

\subsubsection*{Main oracle proxy result in strong norm}

We prove the main bound in strong norm first and provide the necessary technical tools afterwards. The weak-to-strong transfer of error bounds requires at least higher moment bounds, so that we derive immediately results in high probability. From now on, we consider $\tau=\inf\{t\ge t_\circ\,:\,R_t^2\le\kappa\}$ with $t_\circ$ from \eqref{Eqt0}.

\begin{theorem}\label{ThmOrIneqStrong}
Grant \ref{AssBias2}, \ref{AssMon}, \ref{AssVar}, \ref{AssLambdaL2} from Proposition \ref{prop:suffcond} with constants $\pi,C_{V,\lambda},c_\lambda$. Then for  $x\ge 1$ with probability at least $1-c_1e^{-c_2x}$, where $c_1,c_2>0$ are  constants depending on $c_\lambda$ only, we have the oracle-type inequality
\begin{equation}
  \label{eq:mainoracle}
  \norm{\widehat\mu^{(\tau)}-\mu}^2\le K \E\big[\norm{\widehat\mu^{(t^\ast)}-\mu}^2\big]
+2\delta^2x\lambda_{\floor x \wedge D}^{-2},
\end{equation}
with
\[K:= 4C_{V,\lambda}\Big(1+\Big(\frac{(4\sqrt D+12)\delta^2+\sqrt{32}\delta B_{t^\ast,\lambda}(\mu)}{\ind{t^\ast>t_\circ}\min\big(V_{t^\ast,\lambda},B_{t^\ast,\lambda}^2(\mu)\big) + \ind{t^\ast=t_\circ} V_{t_\circ,\lambda}}x\Big)^{1/2}\Big)^{2\pi}.
\]
\end{theorem}

\begin{proof}
We bound $(\gamma_i^{(\tau)} \lambda_i^{-1}Y_i-\mu_i)^2\le 2(1-\gamma_i^{(\tau)})^2\mu_i^2+2(\gamma_i^{(\tau)})^2\lambda_i^{-2}\delta^{2}\eps_i^2$ and we use the fact that  any linear function $f(w_1,\ldots,w_D)=\sum_{i=1}^D w_iz_i$ with $z_1,\ldots,z_D\in\R$ attains its maximum over $1\ge w_1\ge\cdots\ge w_D\ge 0$ at one of the extremal points where $w_i={\bf 1}(i\le k)$, $k\in\{0,\ldots,D\}$ ({\it cf.} also the proof of Lemma \ref{LemConvArg} below).
Under \ref{AssMon}  we thus obtain for $\omega>0$
\begin{align*}
\norm{\widehat\mu^{(\tau)}-\mu}^2&\le 2B_\tau^2+2\delta^2\sum_{i=1}^D \big(\gamma_i^{(\tau)}\big)^2\lambda_i^{-2}\eps_i^{2}\\
&\le 2B_\tau^2+2(1+\omega)V_\tau+2\delta^2\max_{1\ge w_1\ge\cdots\ge w_D\ge 0}\sum_{i=1}^D w_i\lambda_i^{-2}(\eps_i^{2}-1-\omega)\\
&= 2B_\tau^2+2(1+\omega)V_\tau+2\delta^2\max_{k=0,\ldots,D}\sum_{i=1}^k \lambda_i^{-2}(\eps_i^{2}-1-\omega).
\end{align*}
By Lemma \ref{LemStochBounddelta} in Appendix \ref{weighted chi2} below, for $\omega=1$ the last term is bounded by $2\delta^2x\lambda_{\floor x \wedge D}^{-2}$ with probability at least $1-C_1e^{-C_2x}$ with $C_1,C_2>0$ depending only on $c_\lambda$ from \ref{AssLambdaL2}.

For $\tau<t^\ast$ (and so $t^*>t_\circ$) we have $V_\tau\le V_{t^\ast}$, and the bias transfer bound $B_\tau^2\le (B_{\tau,\lambda}^2/B_{t^\ast,\lambda}^2)B_{t^\ast}^2$ is ensured by Lemma \ref{LemBiasTrans} below under \ref{AssBias2}. In addition, Proposition \ref{PropConcBias} guarantees that the weak bias $B_{\tau,\lambda}^2$ is with high probability close to the oracle proxy analogue $B_{t^\ast,\lambda}^2$ under \ref{AssMon}. We deduce more precisely that with probability at least $1-3e^{-x}$
($x\ge 1$):
\[ B_\tau^2\le \Big(1+ \ind{t^\ast > t_\circ}\frac{(4\sqrt D+12)\delta^2+\sqrt{32}\delta B_{t^\ast,\lambda}}{B_{t^\ast,\lambda}^2}x\Big)B_{t^\ast}^2.\]

On the other hand, for $\tau>t^\ast$ we have $B_\tau\le B_{t^\ast}$, and we derive, using \ref{AssVar} on the variance transfer and Proposition \ref{PropConcVar} below on the deviation between $V_{\tau,\lambda}$ and $V_{t^\ast,\lambda}$, that with probability at least $1-3e^{-x}$:
\[ V_\tau \le C_{V,\lambda}\Big(1+\Big(\frac{(4\sqrt D+2)\delta^2+\sqrt{8}\delta B_{t^\ast,\lambda}}{V_{t^\ast,\lambda}}x\Big)^{1/2}\Big)^{2\pi}V_{t^\ast}.\]
Observing $\E\big[\norm{\widehat\mu^{(t^\ast)}-\mu}^2\big]=B_{t^\ast}^2+V_{t^\ast}$, the result follows by taking the maximum of the two previous bounds and simplifying the constants.
\end{proof}

A direct consequence of the preceding result is a moment bound, which has the character of an oracle inequality under mild conditions on the weak bias at the oracle proxy $t^\ast$.

\begin{corollary}\label{CorStrongOrIneq}
In the setting of Theorem \ref{ThmOrIneqStrong} assume $C_1 \ind{t^\ast > t_\circ} V_{t_\circ,\lambda} \le B_{t^\ast,\lambda}^2(\mu)\le C_2\sup_{t>0}V_{t,\lambda}$ for some $C_1,C_2>0$. Then under Assumption {\bf (S)} for a constant $C_{\tau,t^\ast}$ only depending on $C_\circ,C_1,C_2,c_\lambda,\pi,C_{V,\lambda},\nu_-,L$:
\[\E\big[\norm{\widehat\mu^{(\tau)}-\mu}^2\big]\le C_{\tau,t^\ast} \E\big[\norm{\widehat\mu^{(t^\ast)}-\mu}^2\big].
\]
\end{corollary}

\begin{proof}
  Noting $\sup_{t>0}V_{t,\lambda}=D\delta^2$, and on  the one hand
  $V_{t^\ast,\lambda}\ge V_{t_\circ,\lambda}= C_\circ\sqrt D\delta^2
  \ge C_\circ C_2^{-1/2} B_{t^\ast,\lambda}(\mu) \delta$,  on the other hand,
   $B_{t^\ast,\lambda}^2(\mu) \ge C_1 V_{t_\circ,\lambda}$ (then further bounded as above) in the case $t^\ast>t_\circ$, it is simple to check that the factor $K$ in Theorem~\ref{ThmOrIneqStrong} is bounded for $x\ge 1$ as
 $  K \le C_* x^\pi$, where $C_*$ only depends on $C_\circ,C_1,C_2,C_{V,\lambda}$.
  For the remainder term in \eqref{eq:mainoracle},
  note that by Assumption {\bf (S)} for $m\in\N$ and $L^{m-1}<k\le L^m$
\[ \lambda_k^{-2}\le \lambda_{L^m \wedge D}^{-2}\le L^{2m/\nu_-}\lambda_1^{-2}\le (Lk)^{2/\nu_-}\lambda_1^{-2},\]
and thus
\[\delta^2x\lambda_{\floor x \wedge D}^{-2} \le \delta^2x(Lx)^{2/\nu_-}\lambda_1^{-2} \le C_\circ^{-1}L^{2/\nu_-}D^{-1/2} x^{1+2/\nu_-} \lambda_1^{-2}V_{t_\circ,\lambda}, \quad x\ge 1.
\]
  Due to the polynomial increase in $x$ both for $K$ and the remainder term, we can now integrate
  the bound with respect to $c_1e^{-c_2x}dx$, note $\lambda_1^{-2}V_{t_\circ,\lambda}\le
  V_{t_\circ} \leq V_{t^\ast}$ and obtain the announced result.
\end{proof}

The corollary shows that the estimator $\widehat\mu^{(\tau)}$ performs also in strong norm up to a constant as well as $\widehat\mu^{(t^\ast)}$ with the deterministic oracle proxy $t^\ast$. The main restriction is the choice of the minimal index $t_\circ$ according to \eqref{Eqt0}. For smaller $t_\circ$ the variability in the residual and thus in $\tau$ would induce a too high variability in $\widehat\mu^{(\tau)}$, compared to the variance of the oracle estimator.

\subsubsection*{Intermediate estimates from weak to strong norm}

We now set out in detail the ingredients used in the proof of Theorem \ref{ThmOrIneqStrong}.

\begin{lemma} \label{LemBiasTrans}
Under \ref{AssBias2} of Proposition \ref{prop:suffcond}, we have for $t \ge t' \ge t_\circ$ that $B_{t',\lambda}^2\le CB_{t,\lambda}^2$ for some $C\ge 1$ implies $B_{t'}^2\le CB_{t}^2$.
\end{lemma}

\begin{proof}
The assumed decay for the filter ratios implies that there is an index $i_0\in\{0,1,\ldots,D\}$ such that $1-\gamma_i^{(t')}\le C(1-\gamma_i^{(t)})$ holds for $i>i_0$ and $1-\gamma_i^{(t')}\ge C(1-\gamma_i^{(t)})$ for $i\le i_0$ (trivial cases for $i_0=0$, $i_0=D$). Then:
\begin{align*}
B_{t'}^2-CB_t^2
& = \sum_{i=1}^{D} \big((1-\gamma_i^{(t')})^2-C(1-\gamma_i^{(t)})^2\big)\mu_i^{2}\\
 &\le \lambda_{i_0}^{-2}\sum_{i=1}^{i_0} \big((1-\gamma_i^{(t')})^2-C(1-\gamma_i^{(t)})^2\big)\lambda_i^2\mu_i^{2}\\
 &\qquad -\lambda_{i_0}^{-2}\sum_{i=i_0+1}^{D} \big(C(1-\gamma_i^{(t)})^2-(1-\gamma_i^{(t')})^2\big)\lambda_i^2\mu_i^{2}\\
&\le \lambda_{i_0}^{-2}(B_{t',\lambda}^2-CB_{t,\lambda}^2) \le 0,
\end{align*}
which implies the assertion.
\end{proof}

\begin{lemma}\label{LemDevRm}
For any $x>0$ we have with probability at least  $1-2e^{-x}$
\[ \abs{R_{t^\ast}^2-\E[R_{t^\ast}^2]}\le \big(2\delta^2\sqrt D+\sqrt{8}\delta B_{t^\ast,\lambda}\big)\sqrt x +2\delta^2x. \]
\end{lemma}

\begin{proof}
We have $R_{t^\ast}^2-\E[R_{t^\ast}^2]=\sum_{i=1}^D(1-\gamma_i^{(t^\ast)})^2\big(\delta^2(\eps_i^2-1)+2\lambda_i\mu_i\delta \eps_i\big)$. By  Lemma \ref{LemmaLaurentMassart} in the Appendix, $\delta^2\sum_{i=1}^D(\eps_i^2-1)$ is with probability at least $1-e^{-x}$ smaller than $\delta^22\sqrt{ D x}+\delta^2 2x$, while the Gaussian summand is with the same probability smaller than $2\delta B_{t^\ast,\lambda}\sqrt{2x}$, using $(1-\gamma_i^{(t^\ast)})^4\le (1-\gamma_i^{(t^\ast)})^2$.
\end{proof}

\begin{lemma}\label{LemConvArg}
Under \ref{AssMon} of Proposition \ref{prop:suffcond} we have for any  $z_1,\ldots,z_D\in\R$
\begin{align*}
\sum_{i=1}^D\big((1-\gamma_i^{(\tau)})^2-(1-\gamma_i^{(t^\ast)})^2\big)z_i &\le \max_{k=0,\ldots,D}\sum_{i=k+1}^D z_i\text{ on }\{\tau\le t^\ast\},\\
\sum_{i=1}^D\big((1-\gamma_i^{(t^\ast)})^2-(1-\gamma_i^{(\tau)})^2\big)z_i &\le \max_{k=0,\ldots,D}\sum_{i=1}^k (1-\gamma_i^{(t^\ast)})^2z_i\text{ on }\{\tau\ge t^\ast\}.
\end{align*}
\end{lemma}

\begin{proof}
For $\tau\le t^\ast$ introduce the weight space
\[W^\le=\Big\{w\in\R^D\,:\,w_i\in[(1-\gamma_i^{(t^\ast)})^2,1+(1-\gamma_i^{(t^\ast)})^2],\, w_i \text{ increasing in } i\Big\}.\]
Then $\big((1-\gamma_i^{(\tau)})^2\big)_{1 \le i \le D} \in W^\le$ holds on $\{\tau\le t^\ast\}$ by \ref{AssMon} for the monotonicity in $i$, and because of $\gamma_i^{(\tau)}\in[0,\gamma_i^{(t^\ast)}]$.
The set $W^\le$ is convex with extremal points
\[w^{k}=\big((1-\gamma_i^{(t^\ast)})^2+{\bf 1}(i> k)\big)_{1\le i \le D}, \quad k=0,1,\ldots,D.\]
Hence, the linear functional $w\mapsto\sum_i w_iz_i$ attains its maximum over $W^\le$ at some $w^k$. This implies
\[ \sum_{i=1}^D(1-\gamma_i^{(\tau)})^2z_i \le \max_{k=0,\ldots,D}\Big\{\sum_{i=1}^D(1-\gamma_i^{(t^\ast)})^2z_i+\sum_{i=k+1}^D z_i\Big\}\text{ on }\{\tau\le t^\ast\},
\]
which gives the first inequality. For the second inequality consider
\[W^\ge=\Big\{w\in\R^D\,:\,w_i\in[0,(1-\gamma_i^{(t^\ast)})^2],\, w_i \text{ increasing in } i\Big\}\]
and conclude similarly via $\sum_{i=1}^D(1-\gamma_i^{(\tau)})^2z_i \ge \min_{k}\sum_{i=k+1}^D (1-\gamma_i^{(t^\ast)})^2z_i$ on $\{\tau\ge t^\ast\}$.
\end{proof}
Next, we treat the deviation of the weak bias part.
\begin{proposition}\label{PropConcBias}
Under \ref{AssMon} of Proposition \ref{prop:suffcond}, we obtain for any $x\ge 1$ that with probability at least  $1-3e^{-x}$
\[ B_{\tau,\lambda}^2-B_{t^\ast,\lambda}^2\le \Big((4\sqrt D+12)\delta^2+\sqrt{32}\delta B_{t^\ast,\lambda}\Big)x.\]
\end{proposition}

\begin{proof}
  Since $t \mapsto B_{t,\lambda}^2$ is nonincreasing, only the case $\tau<t^\ast$ needs to be considered. By definition of $\tau$, we  obtain $R_\tau^2\le \kappa = \E[R_{t^\ast}^2]$ (since $t^\ast>
  \tau \ge t_0$), and thus, by $\gamma_i^{(\tau)}\le\gamma_i^{(t^\ast)}$:
\begin{multline*}
B_{\tau,\lambda}^2-B_{t^\ast,\lambda}^2  =R_{\tau}^2-R_{t^\ast}^2-\sum_{i=1}^D\big((1-\gamma_i^{(\tau)})^2-(1-\gamma_i^{(t^\ast)})^2\big) \big(\delta^2\eps_i^2+2\lambda_i\mu_i\delta \eps_i\big)\\
\begin{aligned}
&\le \E[R_{t^\ast}^2]-R_{t^\ast}^2-\sum_{i=1}^D\big((1-\gamma_i^{(\tau)})^2-(1-\gamma_i^{(t^\ast)})^2\big) \big(\delta^2\eps_i^2+2\lambda_i\mu_i\delta \eps_i\big)\\
&\le \E[R_{t^\ast}^2]-R_{t^\ast}^2+2\delta\sum_{i=1}^D\big((1-\gamma_i^{(\tau)})^2-(1-\gamma_i^{(t^\ast)})^2\big)(-\lambda_i\mu_i\eps_i).
\end{aligned}
\end{multline*}
By Lemma \ref{LemConvArg}, for any $\omega >0$, the last term  is  bounded as
\begin{align*}
2\delta\sum_{i=1}^D&\big((1-\gamma_i^{(\tau)})^2-(1-\gamma_i^{(t^\ast)})^2\big) (-\lambda_i\mu_i\eps_i)\\
= &  2\delta \sum_{i=1}^D\big((1-\gamma_i^{(\tau)})^2-(1-\gamma_i^{(t^\ast)})^2\big) \big(-\lambda_i\mu_i(\eps_i+\omega \delta^{-1} \lambda_i\mu_i)\big) \\
& \qquad +2\omega(B_{\tau,\lambda}^2-B_{t^\ast,\lambda}^2)\\
\le &  2\delta \max_{k=0,\ldots,D}\sum_{i=k+1}^D \big(-\lambda_i\mu_i\eps_i-\omega \delta^{-1}\lambda_i^2\mu_i^2\big) +2\omega(B_{\tau,\lambda}^2-B_{t^\ast,\lambda}^2).
\end{align*}
Concerning the sum within the maximum, we can identify the term $-\lambda_i\mu_i\eps_i$ with an increment of Brownian motion $B$ over a time step $\lambda_i^2\mu_i^2$. Hence, the maximum is smaller than $\max_{t>0}(B_t-\omega \delta^{-1}t)$ which is exponentially distributed with parameter $2\omega \delta^{-1}$, see Problem 3.5.8 in Karatzas and Shreve \cite{KS}. This term is thus smaller than $\frac{x \delta}{2\omega}$ with probability at least $1-e^{-x}$. In view of Lemma \ref{LemDevRm} we have with probability at least $1-3e^{-x}$, $x\ge 1$,
\[ B_{\tau,\lambda}^2-B_{t^\ast,\lambda}^2\le \Big(2(1+\sqrt D)\delta^2+\sqrt{8}\delta B_{t^\ast,\lambda}+\frac{\delta^2}{\omega}\Big)x+2\omega( B_{\tau,\lambda}^2-B_{t^\ast,\lambda}^2).
\]
The choice $\omega=1/4$ yields the result.
\end{proof}
Finally, for the stochastic error, we  obtain a comparable deviation result.
\begin{proposition}\label{PropConcVar}
Under \ref{AssMon} of Proposition \ref{prop:suffcond}, we obtain for any $x\ge 1$, with probability at least $1-3e^{-x}$:
\[
 V_{\tau,\lambda}^{1/2}-V_{t^\ast,\lambda}^{1/2}\le  \Big(\delta^2(4\sqrt D+2)+\sqrt{8} \delta B_{t^\ast,\lambda}\Big)^{1/2}\sqrt x.
\]
\end{proposition}

\begin{proof}
  Since $t \mapsto V_{t,\lambda}^{1/2}$ is nondecreasing, we only need to consider the case $\tau>t^\ast$. Using $V_{t,\lambda}^{1/2}= \delta \norm{\gamma^{(t)}}$, the inverse triangle inequality, $(A-B)^2\le A^2-B^2$ for $A\ge B\ge 0$, $R_\tau^2\ge \E[R_{t^\ast}^2]$ for $\tau>t^\ast$, and Lemma \ref{LemConvArg}, we obtain:
\begin{multline*}
  \delta^{-2}\big(V_{\tau,\lambda}^{1/2}-V_{t^\ast,\lambda}^{1/2}\big)^2 \le \norm{\gamma^{(\tau)}-\gamma^{(t^\ast)}}^2\\
  \begin{aligned}
&\le\norm{1-\gamma^{(t^\ast)}}^2-\norm{1-\gamma^{(\tau)}}^2\\
&=\delta^{-2}(R_{t^\ast}^2-R_{\tau}^2)+\sum_{i=1}^D\big((1-\gamma_i^{(t^\ast)})^2-(1-\gamma_i^{(\tau)})^2\big)(1-\delta^{-2}Y_i^2)\\
&\le \delta^{-2}(R_{t^\ast}^2-\E[R_{t^\ast}^2])
+ \max_{k=0,\ldots,D}\sum_{i=1}^k (1-\gamma_i^{(t^\ast)})^2(1-\delta^{-2}Y_i^2).
  \end{aligned}
  \end{multline*}
Observe next that $Y_i^2$ is stochastically larger under $P_\mu$ with $\mu_i\not=0$ than under $P_\mu$ with $\mu_i=0$, using the unimodality and symmetry of the normal density:
\begin{align*}
\sup_{\mu\in\R^D} P_\mu(Y_i^2\le y)&=\sup_{\mu_i\in\R}\Big(\Phi(-\lambda_i\mu_i\delta^{-1}+\delta^{-1}\sqrt{y})-\Phi(-\lambda_i\mu_i\delta^{-1}-\delta^{-1}\sqrt{y})\Big)\\
&\le\Phi(\delta^{-1}\sqrt{y})-\Phi(-\delta^{-1}\sqrt{y})=P_0(Y_i^2\le y),\quad y>0.
\end{align*}
By independence of $(Y_i)$, it thus suffices to bound the deviation probability of
\[ \delta^{-2}(R_{t^\ast}^2-\E[R_{t^\ast}^2])
+ \max_{k=0,1,\ldots,D}\sum_{i=1}^k(1-\gamma_i^{(t^\ast)})^2(1-\eps_i^2).
\]
Lemma \ref{LemmaLaurentMassart} in the Appendix gives that the maximum is smaller than $2\sqrt{Dx}$ with probability at least $1-e^{-x}$, and Lemma \ref{LemDevRm} gives the deviation bound for the first term, so that the result follows by insertion.
\end{proof}





\section{Oracle property for early stopping} \label{sec:choice of kappa}

It remains to investigate the relationship of the deterministic oracle proxy $t^\ast$ with the balanced oracles $t_{\mathfrak w},t_{\mathfrak s}$ in \eqref{Eqtw},\eqref{Eqts}, which, of course, depend on the choice of $\kappa>0$ that until now was completely arbitrary. We continue working with $t_0=t_\circ$ from \eqref{Eqt0}.

By definition we have $\E[R_{t^\ast}^2]\le\kappa$ and the weak bias at $t^\ast=t^\ast(\mu)$ satisfies
\[ B_{t^\ast,\lambda}^2(\mu)\le\kappa-\delta^2\sum_{i=1}^D(1-\gamma_i^{(t^\ast)})^2= \kappa-D \delta^2-V_{t^\ast,\lambda}+2\delta^2 \sum_{i=1}^D\gamma_i^{(t^\ast)}
\]
with equality if $t^\ast>t_\circ$. At this stage we exactly require \ref{Assl1l2} of Proposition \ref{prop:suffcond} and obtain
\begin{equation}\label{EqWeakBalanceupper}
  B_{t^\ast,\lambda}^2(\mu)-\big(\kappa-D\delta^2\big)
  \le (2C_{\ell^1,\ell^2}-1)V_{t^\ast,\lambda};
\end{equation}
furthermore,  we also have (since $\gamma_i^{(t)} \in [0,1]$):
\begin{equation}
  B_{t^\ast,\lambda}^2(\mu)-\big(\kappa-D\delta^2\big)
   \ge -V_{t^\ast,\lambda}+2\delta^2 \sum_{i=1}^D(\gamma_i^{(t^\ast)})^2
  = V_{t^\ast,\lambda}, \text{ if } t^\ast>t_\circ. \label{EqWeakBalancelower}
\end{equation}
The larger the choice of $\kappa$, the smaller $t^\ast$ and thus also  $V_{t^\ast,\lambda}$. The control of $B_{t^\ast,\lambda}^2(\mu)$ is not clear because in \eqref{EqWeakBalanceupper} the effects  in $\kappa$ and $V_{t^\ast,\lambda}$ work in opposite directions. Note that for $\kappa\le D\delta^2$, the weak bias part dominates the weak variance at $t^\ast$, in other words $t^\ast\le t_{\mathfrak w}$ holds. A natural choice is therefore $\kappa=D\delta^2$
%
but other choices could be tailored; moreover, the noise variance $\delta^2$ usually needs to be estimated. For these reasons we shall allow for deviations of the form
\begin{equation}\label{Eqkappa}
\abs{\kappa-D \delta^2}\le C_\kappa \sqrt D\delta^2 \text{  for some $C_\kappa>0$.}
\end{equation}
Thanks to the control of $\E[\norm{\widehat\mu^{(\tau)}-\widehat\mu^{(t^\ast)}}_{\mathsf A}^2]$ in Proposition \ref{PropMainBound}, a weakly balanced
oracle inequality can be derived.

\begin{proposition}\label{PropWeakOracleIneq}
Grant \eqref{Eqkappa} for $\kappa$ and \ref{Assl1l2} of Proposition \ref{prop:suffcond}.
Then the following oracle inequality holds in weak norm:
\begin{align*}
\E\big[\norm{\widehat\mu^{(\tau)}-\mu}_{\mathsf A}^2\big]
&\le 2C_{\ell^1,\ell^2} \E\big[\norm{\widehat\mu^{(t_{\mathfrak w})}-\mu}_{\mathsf A}^2\big]+4\big(2C_{\ell^1,\ell^2}+C_\kappa \big)\sqrt D\delta^2.
\end{align*}
\end{proposition}


\begin{proof}
  Consider first the case $t_{\mathfrak w} > t^\ast$. Then $B_{t_{\mathfrak w},\lambda}^2=V_{t_{\mathfrak w},\lambda}$ since $t_{\mathfrak w} > t_\circ$, and we have by monotonicity in $t$
  of $V_{t,\lambda}$:
  \[
  \E[\norm{\widehat\mu^{(t_{\mathfrak w})}-\mu}_{\mathsf A}^2]=B_{t_{\mathfrak w},\lambda}^2+V_{t_{\mathfrak w},\lambda} = 2 V_{t_{\mathfrak w},\lambda} \ge 2 V_{t^\ast,\lambda}.
  \]
Moreover, from inequality \eqref{EqWeakBalanceupper} together with \eqref{Eqkappa}, we have
\begin{equation}
  \label{eq:biasbound}
  B_{t^\ast,\lambda}^2 \le (2C_{\ell^1,\ell^2}-1)V_{t^\ast,\lambda}+C_\kappa\sqrt D\delta^2,
  \end{equation}
  and bringing together the  last two displays yields
  \[
  B_{t^\ast,\lambda}^2 +    V_{t^\ast,\lambda}
  \le C_{\ell^1,\ell^2} \E[\norm{\widehat\mu^{(t_{\mathfrak w})}-\mu}_{\mathsf A}^2]+C_\kappa
  \sqrt D\delta^2.
\]
In the case $t_{\mathfrak w} < t^\ast$, since $B_{t_{\mathfrak w},\lambda}^2\le V_{t_{\mathfrak w},\lambda}$
always holds, by monotonicity in $t$ of $B^2_{t,\lambda}$ we have
\[
\E[\norm{\widehat\mu^{(t_{\mathfrak w})}-\mu}_{\mathsf A}^2]=B_{t_{\mathfrak w},\lambda}^2+V_{t_{\mathfrak w},\lambda} \ge  2 B^2_{t_{\mathfrak w},\lambda} \ge 2 B^2_{t^\ast,\lambda}.
\]
Moreover, from equation \eqref{EqWeakBalancelower} (which holds since in this
case $t^\ast> t_\circ$), together with \eqref{Eqkappa}, we have
\[
  V_{t^\ast,\lambda} \le  B_{t^\ast,\lambda}^2 + C_\kappa\sqrt D\delta^2;
\]
combining the two last displays and using $C_{\ell^1,\ell^2}\ge 1$ yields again
\[
B_{t^\ast,\lambda}^2 +    V_{t^\ast,\lambda}
\le C_{\ell^1,\ell^2} \E[\norm{\widehat\mu^{(t_{\mathfrak w})}-\mu}_{\mathsf A}^2]
+C_\kappa \sqrt D\delta^2,
\]
so that this inequality holds in all cases (including $t^\ast=t_{\mathfrak w}$ in which case the inequality holds trivially since $C_{\ell^1,\ell^2}\ge 1$).
Applying \eqref{eq:propmainbound} and $(A+B)^2\le 2A^2 + 2B^2$, we arrive at
\begin{multline*}
\E\big[\norm{\widehat\mu^{(\tau)}-\mu}_{\mathsf A}^2\big] \le 2\E\big[\norm{\widehat\mu^{(t^\ast)}-\mu}_{\mathsf A}^2\big]+2\big(2D\delta^4+4\delta^2 B_{t^\ast,\lambda}^2\big)^{1/2}\\
\le 2C_{\ell^1,\ell^2}\E[\norm{\widehat\mu^{(t_{\mathfrak w})}-\mu}_{\mathsf A}^2]+2C_\kappa
\sqrt D\delta^2+4\delta\big(\tfrac12D\delta^2+B_{t^\ast,\lambda}^2\big)^{1/2}.
\end{multline*}
Furthermore, $\tfrac12D\delta^2+ B_{t^\ast,\lambda}^2\le (2C_{\ell^1,\ell^2}-\tfrac12)D\delta^2+C_\kappa\sqrt D\delta^2$ follows directly from \eqref{eq:biasbound} (which holds in all cases) and the trivial bound $V_{t^\ast,\lambda}\le D\delta^2$. It remains to simplify the bound, using $C_{\ell^1,\ell^2}\ge 1$.
\end{proof}

In weak norm, the oracle inequality immediately implies rate-optimal estimation by $\widehat\mu^{(\tau)}$ whenever the weak oracle error $\inf_{t\ge 0}\E\big[\norm{\widehat\mu^{(t)}-\mu}_{\mathsf{A}}^2\big]$ is at least of order ${\sqrt D}\delta^2$.
The constants are not optimised, but give a reasonable order of magnitude.

In strong norm, the oracle property is more involved.
The next result shows that the strong error at $t^\ast$ can be bounded by the strong  error at the weakly balanced oracle $t_{\mathfrak w}$, which depends only on the underlying regularisation method and on the spectrum of $\mathsf A$, but not on the particular adaptation method.

\begin{proposition}\label{LemOracleProxyBal}
Grant \ref{AssBias2}, \ref{AssVar} and \ref{Assl1l2} of Proposition \ref{prop:suffcond} and \eqref{Eqkappa} for $\kappa$ with constants $\pi,C_{V,\lambda},C_{\ell^1,\ell^2},C_\kappa$. Then the oracle proxy $t^\ast$ and the weakly balanced oracle $t_{\mathfrak w}$ satisfy the strong norm bound
$$\E\big[\norm{\widehat\mu^{(t^\ast)}-\mu}^2\big]\le C_{t^\ast,t_{\mathfrak w}} \E\big[\norm{\widehat\mu^{(t_{\mathfrak w})}-\mu}^2\big],
$$
where $C_{t^\ast,t_{\mathfrak w}}=\max\Big(2C_{\ell^1,\ell^2}+C_\kappa C_\circ^{-1}-1,C_{V,\lambda}(1+C_\kappa C_\circ^{-1})^\pi\Big)$.
\end{proposition}

\begin{proof}
For $t^\ast<t_{\mathfrak w}$, we  obtain by \eqref{EqWeakBalanceupper}, using $V_{t^\ast,\lambda}\le V_{t_{\mathfrak w},\lambda}=B_{t_{\mathfrak w},\lambda}^2$
(equality due to $t_{\mathfrak w}>t_\circ$) as well as $C_\circ\sqrt D\delta^2\le V_{t_\circ,\lambda}\le V_{t_{\mathfrak w},\lambda}$:
\[ B_{t^\ast,\lambda}^2 \le (2C_{\ell^1,\ell^2}+C_\kappa C_\circ^{-1}-1)B_{t_{\mathfrak w},\lambda}^2.\]
By Lemma \ref{LemBiasTrans}, we can transfer a weak bias inequality into a strong bias inequality with the same constant and the result follows. In the case $t^\ast>t_{\mathfrak w}$, we argue in a similiar manner using \eqref{EqWeakBalancelower} (which holds
since $t^*>t_\circ$):
\[ V_{t^\ast,\lambda}\le B_{t^\ast,\lambda}^2+C_\kappa\sqrt D\delta^2\le  B_{t_{\mathfrak w},\lambda}^2+C_\kappa C_\circ^{-1} V_{t_\circ,\lambda}\le   (1+C_\kappa C_\circ^{-1})V_{t_{\mathfrak w},\lambda},\]
followed by the variance transfer guaranteed by \ref{AssVar}.
\end{proof}

Next, we turn to the control of the strong bias at the weak oracle $t_{\mathfrak w}$. Surprisingly, this is quite universally feasible whenever $t_{\mathfrak w}$ is smaller than $t_{\mathfrak s}$.

\begin{proposition}\label{PropOrProperty1}
  Grant Assumptions {\bf (R)} and {\bf (S)}. For all $\mu$ with $t_{\mathfrak w}(\mu) \le t_{\mathfrak s}(\mu)$ 
  we have with the constant $C_V$ from Lemma \ref{LemVart}:
 \[\E\big[\norm{\widehat\mu^{({t_{\mathfrak w}})}-\mu}^2]\le C_{\mathfrak w,\mathfrak s} \E\big[\norm{\widehat\mu^{(t_{\mathfrak s} )}-\mu}^2\big],
 \]
 with $C_{\mathfrak w,\mathfrak s}=\big((2\beta_+)^{2/\rho}C_V+4\big)$.
\end{proposition}

\begin{proof}
 First assume $t_{\mathfrak s} \leq \lambda_{D}^{-1}$.
By Assumption~\ref{AssGComp} we have $4(1-\gamma_i^{(t)})^2\ge 1$ if $t\lambda_i\le c:=(2\beta_+)^{-1/\rho}$. Consequently, for any $t\geq t_{\mathfrak w}$:
\begin{align*}
B_{t_{\mathfrak w}}^2-4B_t^2 &= \sum_{i=1}^D \big((1-\gamma_i^{(t_{\mathfrak w})})^2-4(1-\gamma_i^{(t)})^2\big)\mu_i^2\\
&\le \sum_{i:\lambda_i> ct^{-1}}  (1-\gamma_i^{(t_{\mathfrak w})})^2\mu_i^2\\
&\le \sum_{i:\lambda_i> ct^{-1}}  (1-\gamma_i^{(t_{\mathfrak w})})^2(c^{-1}t\lambda_i)^2\mu_i^2\\
& \le c^{-2}t^2  B_{t_{\mathfrak w},\lambda}^2\le c^{-2}t^2  V_{t_{\mathfrak w},\lambda}.
\end{align*}
From Lemma \ref{LemVart} we know $V_{t,\lambda}\le C_V(t\wedge \lambda_D^{-1})^{-2}V_t$.
We insert $t=t_{\mathfrak s} (\leq \lambda_D^{-1})$ and use
$V_{t_{\mathfrak w},\lambda}\le V_{{\mathfrak s},\lambda}$ to conclude
\[ B_{t_{\mathfrak w}}^2\le 4B_{t_{\mathfrak s}}^2+c^{-2}C_V V_{t_{\mathfrak s}}.\]
Adding $V_{t_{\mathfrak w}}\le V_{{\mathfrak s}}$ and simplifying the constant yields the result.

Consider now the case $t_{\mathfrak s} > \lambda_D^{-1}$.
In this case Lemma~\ref{LemVart} implies $\lambda_D^{-2} V_{t_{\mathfrak s,\lambda}} \leq C_V  V_{t_{\mathfrak s}} $.
For the bias, we have by definition of $t_{\mathfrak w}$ and monotonicity:
\[
  B^2_{t_{\mathfrak w}} \le \lambda_D^{-2} B^2_{t_{\mathfrak w},\lambda}
  \le \lambda_D^{-2} V_{t_{\mathfrak w},\lambda}  \leq \lambda_D^{-2} V_{t_{\mathfrak s},\lambda}
  \leq C_V V_{t_{\mathfrak s}},
\]
also implying the desired result.
\end{proof}

Section \ref{SecLandweberSignal} below shows for the Landweber method that $t_{\mathfrak w}(\mu) \le t_{\mathfrak s}(\mu)$ or at least $V_{t_{\mathfrak w}(\mu)} \lesssim  V_{t_{\mathfrak s}(\mu)}$ holds for a large class of polynomially decaying signals $\mu$. For rapidly decaying signals $\mu$, however, the inverse relationship $ t_{\mathfrak w}(\mu)  \gg t_{\mathfrak s}(\mu)$ may happen:

\begin{example}[Generic counterexample to $t_{\mathfrak w} \le t_{\mathfrak s}$]\label{CounterEx}
  Consider the signal $\mu_1=1$, $\mu_i=0$ for $i\ge 2$ and assume $\lambda_1=1$, $\gamma_1^{(t)}<1$  for all $t\ge 0$. Then we have $B_t^2(\mu)=B^2_{t,\lambda}(\mu)>0$ whereas $V_t\thicksim (t\wedge \lambda_D^{-1})^2V_{t,\lambda}$ holds in the setting of Lemma \ref{LemVart}. Hence, noting that $t_{\mathfrak w}\to\infty$  as $\delta\to 0$, we see that  $V_{t_{\mathfrak w}}
  \thicksim  ( t_{\mathfrak w} \wedge \lambda_D^{-1})^2 V_{t_{\mathfrak w},\lambda}
  = (t_{\mathfrak w} \wedge \lambda_D^{-1}) B^2_{t_{\mathfrak w},\lambda }(\mu)
  = (t_{\mathfrak w } \wedge \lambda_D^{-1})^2 B_{t_{\mathfrak w}}^2(\mu)$ is
  larger than $B_{t_{\mathfrak w}}^2(\mu)$, implying $t_{\mathfrak w} > t_{\mathfrak s}$.
  If we consider an asymptotic setting where $D\rightarrow \infty, \lambda_D \rightarrow 0$
  as $\delta \rightarrow 0$, we even have
  $t_{\mathfrak w}/t_{\mathfrak s}\to\infty$ as $\delta\to 0$.
\end{example}

The weakly balanced oracle does not profit from the regularity of $\mu$ in strong norm. Notice that this loss of efficiency is  intrinsic to residual-based stopping rules which have access to the weak bias only. Still, we are able to control the error by an inflated weak oracle error.

\begin{theorem} \label{thm: oracle end}
  Suppose Assumptions {\bf (R)}, {\bf (S)} hold with $\rho>\max( \nu_+,1+\frac{\nu_+}{2})$
  and  \eqref{Eqkappa}  holds for $\kappa$ with $C_\kappa\in[0,C_\circ)$. Then for all $\mu$ with $ t_{\mathfrak w}(\mu)  \le t_{\mathfrak s}(\mu)$ we  have
\[\E\big[\norm{\widehat\mu^{(\tau)}-\mu}^2\big]\le C_{\tau,{\mathfrak s}}\E\big[\norm{\widehat\mu^{(t_{\mathfrak s})}-\mu}^2\big].
\]
For all $\mu$ with 
$t_{\mathfrak w}(\mu)  \ge t_{\mathfrak s}(\mu)$ we obtain
\[\E\big[\norm{\widehat\mu^{(\tau)}-\mu}^2\big]\le C_{\tau,{\mathfrak w}}
(t_{\mathfrak w}\wedge \lambda_D^{-1})^2\E\big[\norm{\widehat\mu^{(t_{\mathfrak w})}-\mu}_{\mathsf A}^2\big].\]
The constants $C_{\tau,{\mathfrak s}}$ and
$C_{\tau,{\mathfrak w}}$ depend only on $\rho,\beta_-,\beta_+,L,\nu_-,\nu_+,C_\circ,C_\kappa$.
\end{theorem}

\begin{proof}
  We want to apply Corollary \ref{CorStrongOrIneq} (bounding the strong error of $\wh{\mu}^{(\tau)}$
  by that of $\wh{\mu}^{(t^\ast)}$) followed by Proposition \ref{LemOracleProxyBal}
  (from $\wh{\mu}^{(t^\ast)}$ to $\wh{\mu}^{(t_{\mathfrak w})}$) in order to bound $\E[\norm{\widehat\mu^{(\tau)}-\mu}^2]$ by
  $C_{\tau,t^\ast} C_{t^\ast,t_{\mathfrak w}}
  \E\big[\norm{\widehat\mu^{(t_{\mathfrak w})}-\mu}^2\big].
$

  We have that $C_\kappa\in[0,C_\circ)$ implies by the bounds in \eqref{EqWeakBalanceupper}-\eqref{EqWeakBalancelower} together with $V_{t_\ast,\lambda}\ge V_{t_\circ,\lambda}=C_\circ\sqrt D \delta^2$ that $B_{t^\ast,\lambda}^2\in [(1-C_\kappa C_\circ^{-1})V_{t^\ast,\lambda}\ind{t^\ast>t_\circ},(2C_{\ell^1,\ell^2}+C_\kappa C_\circ^{-1}-1)V_{t^\ast,\lambda}]$, as required for Corollary~\ref{CorStrongOrIneq}. Also the condition $\rho > \nu_+$
ensures via Proposition~\ref{prop:suffcond} that \ref{AssBias2}, \ref{AssMon}, \ref{AssVar}, \ref{AssLambdaL2} hold as required.

For the case $t_{\mathfrak w}  \le t_{\mathfrak s}$ we can  conclude the first inequality by the bound on $\E[\norm{\widehat\mu^{(t_{\mathfrak w})}-\mu}^2]$ in Proposition \ref{PropOrProperty1}. In the other case, $B_{t_{\mathfrak w}}^2\le V_{t_{\mathfrak w}}\le  c_V^{-1}(t_{\mathfrak w} \wedge \lambda_D^{-1})^2V_{t_{\mathfrak w},\lambda}$ is implied by Lemma \ref{LemVart}, using $\rho>1+\nu_+/2$, and the second result follows. It remains to trace back the dependencies of the constants involved.
\end{proof}

Let us specify this main result for  polynomially decaying singular values $\lambda_i$. Then we can write an oracle inequality which involves the oracle errors in weak and strong norm instead of the index $t_{\mathfrak w}$ itself.

\begin{corollary}\label{CorTauws}
Grant Assumption {\bf (R)} with $\rho>1$, \eqref{Eqkappa} with $C_\kappa\in[0,C_\circ)$ and $\lambda_i= c_Ai^{-1/\nu}, i=1,\ldots,D$, for $c_A>0$ and $0<\nu<\min(\rho,2\rho-2)$. Then
\[\E\big[\norm{\widehat\mu^{(\tau)}-\mu}^2\big]\le  C_{\tau,\mathfrak{ws}}\max\Big(\delta^{-4/\nu}\E\big[\norm{\widehat\mu^{(t_{\mathfrak w})}-\mu}_{\mathsf A}^2\big]^{1+2/\nu},\,\E\big[\norm{\widehat\mu^{(t_{\mathfrak s})}-\mu}^2\big]\Big)
\]
holds with a constant $C_{\tau,\mathfrak{ws}}$ depending only on $\rho,\beta_-,\beta_+,c_A,C_\circ,C_\kappa$.
\end{corollary}

\begin{proof}
  Since Assumption~{\bf (S)} holds with $\nu_+=\nu_-=\nu$, by Lemma~\ref{le:ineqlambda1} in the Appendix it  holds that $t_{\mathfrak w} \geq t_\circ \geq \zeta \lambda_1^{-1} \geq \zeta$ ($\zeta$ only depending on $C_\circ$ and $L$).
From $\delta^{-2}V_{t_{\mathfrak w},\lambda}= \sum_{i=1}^D (\gamma_i^{(t_{\mathfrak w})})^2$ we deduce via Assumption \ref{AssGComp}
\begin{align*}
  \delta^{-2}V_{t_{\mathfrak w},\lambda}
  &\ge \beta_-^2 ((\zeta c_A)^\rho \wedge 1)^2
    \#\{i: 1 \leq i \leq D ,
    \lambda_i\ge \zeta c_A /t_{\mathfrak w}\}\\
  &\ge \beta_-^2 ( (\zeta c_A)^\rho \wedge 1)^2
    (\floor{(\zeta^{-1} t_{\mathfrak w})^{\nu}} \wedge D)\\
  &  \ge \beta_-^2 ( (\zeta c_A)^{\rho} \wedge 1)^2 ( (\zeta^{-1} t_{\mathfrak w})^{\nu}/2 \wedge
    (c_A \lambda_D^{-1})^{\nu})\\
& \ge C_1
 (t_{\mathfrak w}\wedge \lambda_D^{-1})^{\nu}, \text{ with }
      C_1 = \frac{1}{2} \beta_-^2 \zeta^{2\rho} ( \zeta^{-1} \wedge c_A)^{\nu+2\rho}.
\end{align*}
Because of $\delta^{-4/\nu}\E[\norm{\widehat\mu^{(t_{\mathfrak w})}-\mu}_{\mathsf A}^2]^{2/\nu}\ge (\delta^{-2}V_{t_{\mathfrak w},\lambda})^{2/\nu}\ge C_1^{2/\nu}(t_{\mathfrak w}\wedge \lambda^{-1}_D)^{2}$, the  bound follows by combining the two inequalities from Theorem \ref{thm: oracle end}.
\end{proof}

A further consequence is a minimax rate-optimal bound over the Sobolev-type ellipsoids
\begin{equation}\label{EqHbeta}
H_d^\beta(R)=\Big\{\mu\in\R^D:\sum_{i=1}^D i^{2\beta/d}\mu_i^2\le R^2\Big\}, \;\;\beta\ge 0, R>0.
\end{equation}
In the case of Fourier coefficient sequences $(\mu_i)$ the class $H_d^\beta(R)$ corresponds to a ball of radius $R$ in a $d$-dimensional $L^2$-Sobolev space of regularity $\beta$.
At this stage the concept of qualification of the spectral regularisation method, {\em i.e.}, the filter sequence, enters.

\begin{corollary}\label{CorTauMinimax}
Grant Assumption {\bf (R)} with $\rho>1$, \eqref{Eqkappa} with $C_\kappa\in[0,C_\circ)$, $\lambda_i= c_Ai^{-p/d}$ for $c_A>0$ and $p/d>\min(\rho, (2\rho-2))^{-1}$ as well as
\[ 1-g(t,\lambda)\le C_q (t\lambda)^{-2q},\quad t\ge t_\circ,\,\lambda\in(0,1],\]
for some {\it qualification index} $q>0$. Then $\widehat\mu^{(\tau)}$ attains  the minimax-optimal rate over $H_d^\beta(R)$ from \eqref{EqHbeta}
\[\sup_{\mu\in H_d^\beta(R)}\E\big[\norm{\widehat\mu^{(\tau)}-\mu}^2\big]\lesssim R^2(R^{-1}\delta)^{4\beta/(2\beta+2p+d)},
\]
provided $2q-1\ge \beta/p$ and $(R/\delta)^{2d/(2\beta+2p+d)}\gtrsim  \sqrt D$ for $D\to\infty$ as $\delta\to 0$.
\end{corollary}

\begin{proof}
 A qualification $q\ge \beta/(2p)$ in combination with Assumption {\bf (R)} ensures for $\mu\in H_d^\beta(R)$, compare also Thm. 4.3 in Engl {\it et al.} \cite{EHN}:
\begin{align*}
B_t^2(\mu) &\le \sum_{i=1}^D C_q^2\min\big((t\lambda_i)^{-2q},1\big)^2\mu_i^2
\le C_q^2\sum_{i=1}^D (tc_Ai^{-p/d})^{-2\beta/p}\mu_i^2\\
&\le C_q^2c_A^{-2\beta/p}R^2t^{-2\beta/p}.
\end{align*}

Similarly, we deduce for $2q-1\ge \beta/p$:
\[ t^2B_{t,\lambda}^2(\mu)\le \sum_{i=1}^D C_q^2\min\big((t\lambda_i)^{-2q+1},t\lambda_i\big)^2\mu_i^2\le C_q^2c_A^{-2\beta/p}R^2t^{-2\beta/p}.
\]
Under Assumption \ref{AssGComp} the weak variance satisfies $V_{t,\lambda}\lesssim\delta^2\sum_i\min(\beta_+ (t\lambda_i)^{2\rho},1)
\lesssim \delta^2t^{d/p}$ provided $\lambda_i\thicksim i^{-p/d}$, $p>d/(2\rho)$. For $p>d/(2\rho-2)$ we obtain in a similar manner $V_t\lesssim \delta^2t^{2+d/p}$.

A rate-optimal choice of $t$ is thus of order $(R/\delta)^{2p/(2\beta+2p+d)}$ and  gives
\[\inf_{t\ge 0}\max\Big(\E[\norm{\widehat\mu^{(t)}-\mu}^2\big],\,t^2\E[\norm{\widehat\mu^{(t)}-\mu}_{\textsl A}^2\big]\Big)\lesssim R^2(R^{-1}\delta)^{4\beta/(2\beta+2p+d)},
\]
with a constant independent of $\mu$, $D$ and $\delta$. Using the last assumption in the corollary, we deduce $t\ge t_\circ$ for a rate-optimal choice of $t$ via
\[ V_{t,\lambda}\thicksim \delta^2(R/\delta)^{2d/(2\beta+2p+d)}\gtrsim \delta^2D^{1/2}\thicksim V_{t_\circ,\lambda}.\]
In view of the narrow sense oracle property \eqref{EqBalNarOracle}
of $t_{\mathfrak w}$ and equally of $t_{\mathfrak s}$ in strong norm, we thus conclude by applying Theorem \ref{thm: oracle end}.
\end{proof}

For the filters of Example \ref{detailed filters} we see that the Landweber and Showalter method have any qualification $q>0$ while standard Tikhonov regularisation has qualification $q=1$. The  statement is very much in the spirit of the results for the deterministic discrepancy principle, see {\it e.g.} Thm. 4.17 in Engl {\it et al.} \cite{EHN}, when interpreting $\kappa=D\delta^2$ as the squared noise level, {\it cf.} Hansen \cite{Ha}.  Note, however, that we do not require a slightly enlarged critical value and that the condition $(R/\delta)^{2d/(2\beta+2p+d)}\gtrsim  \sqrt D$, which  means that the minimax estimation rate for $\delta\to 0$ is not faster than $ D^{1/2+p/d}\delta^2$, indicates an intrinsic difference between deterministic and statistical inverse problems.

Cohen {\it et al.} \cite{CHR} argue that in the present setting a dimension $D_\delta\thicksim \delta^{-2d/(2p+d)}$ of the approximation space suffices to attain the optimal rates for all Sobolev (infinite) sequence spaces, i.e., the error introduced by the approximation space of dimension $D_\delta$ is of order smaller than
the minimax rate. With this choice of $D_\delta$, all optimal squared error rates of size $\delta$ and slower are attained by $\widehat\mu^{(\tau)}$; only excluded is the faster rate interval $[\delta^2,\delta)$, which corresponds to the high regularity $\beta>p+d/2$, see also the numerical results and the discussion on the two-stage procedure in Section \ref{discussion and numerics} below.

Theorem \ref{thm: oracle end} can also serve to deduce bounds on the Bayes risk of $\widehat\mu^{(\tau)}$ with respect to a prior  for the signals $\mu$. For concrete methods and general classes of priors   Bayesian oracle inequalities are thus conceivable similar to Bauer and Rei{\ss} \cite{BR2}, but in a different setup.

\section{More on the Landweber method} \label{discussion and numerics}

\subsection{A sufficient condition for the complete oracle property}\label{SecLandweberSignal}

For the concrete example of the Landweber method, let us investigate for which signals $\mu$ we have a true oracle inequality in the sense that in Theorem \ref{thm: oracle end} the first inequality applies with a universal constant $C_{\tau,{\mathfrak s}}>0$. Note first that, under Assumption {\bf (S)},
if we can ensure additionally that $V_{t_{\mathfrak w}} \le C_1 V_{t_{\mathfrak s}}$ for some $C_1>1$, then
Proposition \ref{PropOrProperty1} yields the more general  bound
\begin{equation}\label{EqMuBound}
\E\big[\norm{\widehat\mu^{({t_{\mathfrak w}})}-\mu}^2]\le \max\big((2\beta_+)^{2/\rho}C_V+4,C_1\big)\E\big[\norm{\widehat\mu^{(t_{\mathfrak s})}-\mu}^2\big].
\end{equation}
This generalisation is just due to $B_{t_{\mathfrak w}}(\mu) \le B_{t_{\mathfrak s}}(\mu)$ in the case $t_{\mathfrak w}> t_{\mathfrak s}$.

To establish $V_{t_{\mathfrak w}} \le C_1 V_{t_{\mathfrak s}}$, let us assume $D/L\in\N$ and consider  for some $c_\mu>0$ the  class of signals $\mu$
\begin{equation}\label{EqMuDecay}
{\mathcal C}:={\mathcal C}(L,c_\mu):=\Big\{\mu\in\R^D\,:\;\forall i=1,\ldots,D/L:\,\mu_{(L-1)i+1}^2+\cdots+\mu_{Li}^2\ge c_\mu\mu_{i}^2\Big\}.
\end{equation}
From the definition of the Landweber filters in Example \ref{detailed filters} we obtain
for $i\leq D/L$:
\[ \frac{(1-\gamma_{i}^{(t)})^2}{(1-\gamma_{Li}^{(t)})^2}=\Big(1-\frac{\lambda_i^2-\lambda_{Li}^2}{1-\lambda_{Li}^2}\Big)^{2t^2}
\le \exp\Big(-2t^2\lambda_i^2(1-L^{-2/\nu_+})\Big).
\]
By the decay of $C(r):=re^{-(1-L^{-2/\nu_+})r}$ for $r\ge r_0:=(1-L^{-2/\nu_+})^{-1}$, we can thus bound for $\mu\in{\mathcal C}$
\[ \sum_{\substack{i: t\lambda_i\ge r^{1/2},\\ i \leq D/L}}(1-\gamma_i^{(t)})^2(t\lambda_i)^2\mu_i^2\le
  c_\mu^{-1} C(r)^2 \sum_{\substack{i: t\lambda_i\ge r^{1/2},\\ i \leq D/L}}
  (1-\gamma_{Li}^{(t)})^2\big(\mu_{(L-1)i+1}^2+\cdots+\mu_{Li}^2\big).
\]
  Additionally, if $t \leq \lambda_D^{-1}$, then
  \[
  \sum_{i= D/L +1 }^D (1-\gamma_i^{(t)})^2(t\lambda_i)^2\mu_i^2\le L^{2/\nu_-}   \sum_{i= D/L +1 }^D (1-\gamma_i^{(t)})^2 \mu_i^2.
\]
This implies $t^2B_{t,\lambda}^2(\mu)\le ( r + c_\mu^{-1} C(r)^2  + L^{2/\nu_-} )B_t^2(\mu)$
provided $t \leq \lambda_D^{-1}$. Using $(t\wedge \lambda_D^{-1})^2V_{t,\lambda}\ge C_VV_t$ from Lemma \ref{LemVart}, the definition of the balanced oracles, and monotonicity,
we conclude in the case $t_{\mathfrak{w}} > t_{\mathfrak{s}}$ and $t_{\mathfrak s} \leq \lambda_D^{-1}$:
\begin{align*}
  V_{t_{\mathfrak w}}\le C_V^{-1} (t_{\mathfrak w}\wedge \lambda_D^{-1})^2 V_{t_{\mathfrak w},\lambda}
  &= C_V^{-1} (t_{\mathfrak w}\wedge \lambda_D^{-1})^2 B^2_{t_{\mathfrak w},\lambda}(\mu)\\
  &\le C_V^{-1} (t_{\mathfrak w}\wedge \lambda_D^{-1})^2 B^2_{t_{\mathfrak w}\wedge \lambda_D^{-1},\lambda}(\mu)\\
  & \le C_V^{-1}(c_\mu^{-1} C(r)^2+r  + L^{2/\nu_-})B_{t_{\mathfrak w}\wedge \lambda_D^{-1}}^2(\mu)\\
  & \le C_V^{-1}(c_\mu^{-1} C(r)^2+r  + L^{2/\nu_-})B_{t_{\mathfrak s}}^2(\mu)\\
 & =  C_V^{-1}(c_\mu^{-1} C(r)^2+r  + L^{2/\nu_-})V_{t_{\mathfrak s}}^2,
\end{align*}
providing the desired strong variance inequality. Finally, in the case
$t_{\mathfrak{w}} > t_{\mathfrak{s}} > \lambda_D^{-1}$ we have $\gamma_i^{(t_{\mathfrak s})} \geq \beta_-$ for
all $i=1,\ldots,D$  by Assumption~{\bf (R3)} and therefore
 $V_{t_\mathfrak{w}} \leq \beta_-^{-1} V_{t_\mathfrak{s}}$.

The value of $r$ may be optimised or we just take $r=r_0$ to define $C_1=C(r_0)$ and thus conclude from Theorem \ref{thm: oracle end} and \eqref{EqMuBound}:

\begin{corollary}\label{CorLWOrineq}
In the setting of Theorem \ref{thm: oracle end}, 
we have for Landweber iteration and all signals $\mu\in {\mathcal C}(L,c_\mu)$ from \eqref{EqMuDecay} the oracle inequality
\[\E\big[\norm{\widehat\mu^{(\tau)}-\mu}^2\big]\le C_{\tau,{\mathfrak s}}\E\big[\norm{\widehat\mu^{(t_{\mathfrak s})}-\mu}^2\big]
\]
with a constant $ C_{\tau,{\mathfrak s}}$ depending only on $\rho,\beta_-,\beta_+,L,\nu_-,\nu_+,c_\mu,C_{\kappa},C_{\circ}$.
\end{corollary}

The class $\mathcal C$ in combination with  Counterexample \ref{CounterEx}
illustrates that the early stopping rule $\tau$ may exhibit bad strong norm adaptation  only if the signal has a significantly stronger strength in the lower than in the higher SVD coefficients. Interestingly, also for noise level-free posterior regularisation methods, these kinds of signals must be excluded in deterministic inverse problems, {\it cf.} Prop. 4.6 in Kindermann \cite{Ki}. Let us emphasize, however, that this discussion only concerns the individual oracle approach, whereas the minimax optimality under standard polynomial source conditions is satisfied by Corollary \ref{CorTauMinimax}.

\subsection{Numerical examples}

Consider the moderately ill-posed case $\lambda_i=i^{-1/2}$ (as, {\it e.g.}, for the Radon transform) with noise level $\delta=0.01$ and dimension $D=10\,000$. After 51 Landweber iterations the weak variance attains the level $\sqrt{2 D}\delta^2$, which is the dominating term in \eqref{oracleexact} and corresponds to $C_\circ=\sqrt 2$ in the choice of $t_\circ$ (by abuse of notation, indices $t$ denote numbers of iterations in this subsection).

\begin{figure}[t]
\includegraphics[height=0.2\textheight]{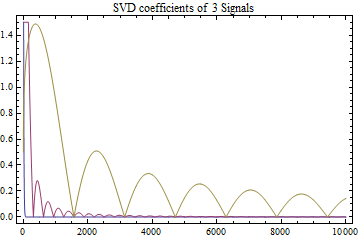} \hspace{5mm}
\includegraphics[height=0.2\textheight]{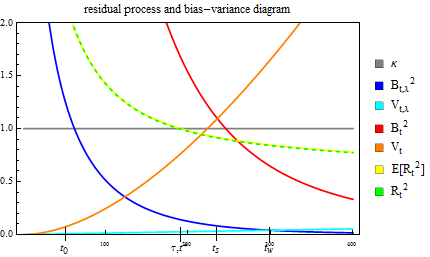}
\caption{Left: SVD representation of a super-smooth (blue), a smooth (red) and a rough (olive) signal. Right: residuals, weak/strong bias and variance for $t\in \{0,\ldots,400\}$ Landweber iterations and the smooth signal.
}\label{Fig2}

\end{figure}

In Figure \ref{Fig2} (left) we see the SVD representation of three signals: a very smooth signal $\mu(1)$, a smooth signal $\mu(2)$ and a rough signal $\mu(3)$, the attributes coming from the interpretation via the decay of Fourier coefficients.  The weakly balanced oracle indices $t_{\mathfrak w}$ are $(42,299,1074)$. So, we stop before $t_\circ$ in the super-smooth case and expect a high variability of $\tau$ around $t^\ast$. The strong indices $t_{\mathfrak s}$ are $(29,235,1185)$.  For the smooth and super-smooth case with rapid decay and oscillations the signals lie in the class $\mathcal C$ from \eqref{EqMuDecay} for relatively small $c_\mu>0$ only, which explains why here $t_{\mathfrak w}>t_{\mathfrak s}$ occurs.

We choose $\kappa=D\delta^2=1.0$ although the ratio $\sum_i\gamma_i^{(t)}/\sum_i(\gamma_i^{(t)})^2$, defining the constant $C_{\ell^1,\ell^2}$, is about $1.05$ at $t_\circ$. The relationship \eqref{EqWeakBalanceupper} therefore indicates that a smaller value of $\kappa$ might be beneficial, which indeed in practice yields much better results, especially for the smooth and rough cases. In any case, in this simulation we use $\kappa=D\delta^2$ off the shelf and we compute the stopping rule $\tau$ starting at $t_0=0$ to illustrate the effects when very early stopping is recommended.

For the smooth signal $\mu(2)$ Figure \ref{Fig2} (right) displays  squared bias, variance and residuals (for one realisation) as a function of the number of iterations  and indicates the stopping indices (the  bias and variance functions are  rescaled to fit into the picture).  We see that the residual and its expected value are hardly distinguishable and thus $\tau$ is very close to the oracle proxy $t^\ast$. While our error analysis uses the weak norm oracle inequality to establish the strong norm bound, here the strongly balanced oracle  $t_{\mathfrak s}$ is closer to $\tau$ than the weak counterpart $t_{\mathfrak w}$. The effects will become obvious in the error plots below.

Relative to the target $t_{\mathfrak w}$, Figure \ref{Fig3} (left) displays box plots (a box representing the inner quartile range, whiskers the main support, points outliers and a horizontal bar the mean) for the relative number of iterations $\tau/t_{\mathfrak w}$ and $\tau/t_{\mathfrak s}$, respectively, in 1000 Monte Carlo repetitions and for the three signals.
We see the high variability  for the super-smooth signal and the fact that for all three signals $\tau$ usually stops too early. For the weak norm this can be cured by choosing $\kappa$ smaller taking into account the size of  $C_{\ell^1,\ell^2}$, as discussed above (independently of the unknown signal). Relative to the strongly balanced oracle $t_{\mathfrak s}$, however, the stopping rule $\tau$ comes closer to the oracle for the super-smooth and smooth signals and moves further away for the rough signal because of $t_{\mathfrak w}>t_{\mathfrak s}$ and  $t_{\mathfrak w}<t_{\mathfrak s}$, respectively.

\begin{figure}[t]
\includegraphics[width=0.48\textwidth]{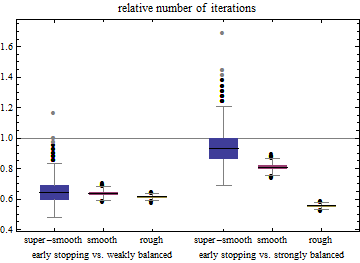}
\includegraphics[width=0.48\textwidth]{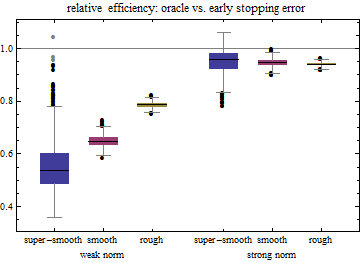}
\caption{Left: number of Landweber iterations for $\tau$ divided by the balanced oracle numbers.
Right:  oracle RMSE errors divided by RMSE of $\hat\mu^{(\tau)}$.}\label{Fig3}
\end{figure}

In Figure \ref{Fig3} (right) the box plots show for the same Monte Carlo run the relative errors $\min_{t\ge 0}\E[\norm{\widehat\mu^{(t)}-\mu}_{\mathsf A}^2]^{1/2}/\E[\norm{\widehat\mu^{(\tau)}-\mu}_{\mathsf A}^2]^{1/2}$ (weak norm) and $\min_{t\ge 0}\E[\norm{\widehat\mu^{(t)}-\mu}^2]^{1/2}/\E[\norm{\widehat\mu^{(\tau)}-\mu}^2]^{1/2}$ (strong norm) such that higher values indicate better performance. As expected from the stopping rule results, the variability is largest for the super-smooth case, and as predicted by theory, the relative efficiency in weak norm is best for the rough signal. In strong norm all three efficiencies become much better. Due to $t_{\mathfrak s}<t_{\mathfrak w}$ this was to be expected for the super-smooth and smooth signals, {\it cf.} also Figure \ref{Fig2} (right). For the rough signal this is quite surprising because the stopping rule $\tau$ usually stops at about half the value of $t_{\mathfrak s}$. This is explained by a roughly linear decay of $t\mapsto B_t^2$ around $t_{\mathfrak s}$ and a strong growth in $t\mapsto V_t$ such that the total error $t\mapsto B_t^2+V_t$ is quite flat to the left of $t_{\mathfrak s}$. Note that in weak norm $t\mapsto B_{t,\lambda}^2$ decays much faster and the 'flatness' disappears. This is the common phenomenon that the {\it relative} error due to a data-driven parameter choice often becomes smaller for ill-posed than for well-posed settings.

Further unreported simulations confirm these findings, in particular the relative error due to residual-based stopping remains small (rarely larger than $2$, {\it i.e.} relative efficiency usually larger than $0.5$). Only for super-smooth signals, where we ought to stop before $t_\circ$, the variability may become harmful.

As a practical procedure, we propose to run the iterations always until $t_\circ$ (51 iterates here) and if the stopping rule $\tau$ tells us not to continue, then we apply a standard model selection procedure to choose among the $t_\circ$ first iterates. For the truncated SVD case this two-step procedure can be proven to be rate-optimal for all regularities $\beta$, not only within the adaptation interval, see \cite{ourpapercutoff}.
For our general spectral regularisation approach the lack of a complete oracle inequality in strong norm is due to potentially stopping
{\em later} than at $t_{\mathfrak s}$. Hence, a natural suggestion would be to apply a second
model selection step
to select among $\widehat\mu^{(0)},\ldots,\widehat\mu^{(\tau)}$ for any outcome of $\tau$. The performance of this general two-step approach needs to be studied further, but seems very promising from the perspective of both, statistical errors and numerical complexity.


\section{Appendix}  \label{Appendix}

\subsection{Proof of Proposition \ref{prop:suffcond}} \label{proof:prop:suffcond}

We start with an important result  for a nonincreasing sequence satisfying {\bf {(S)}}. This is related to comparisons between a function and its power integrals,  also known as {\em Karamata's one-sided relations}.

\begin{lemma}[One-sided Karamata relations] \label{One-sided Karamata relations}
Suppose Assumption {\bf (S)} is satisfied. Then for any $p>0$ and $k \ge 1$ we have
\begin{align}
  \label{eq:karinfnegpbis}
  \frac{\sum_{j\leq k} \lambda_j^{-p}}{k\lambda_{k}^{-p}} & \ge L^{-p/\nu_-}(1-L^{-1}) &&\text{for $p>0$ and $1\le k\le D$,}\\
  \label{eq:karinfpospbis}
  \frac{\sum_{j=k+1}^{Lk} \lambda_j^{p}}{k \lambda_k^{p}} &\ge (L-1) L^{-p/\nu_-} &&\text{for $p>0$ and $1\le k\le \lfloor D/L \rfloor $,}\\
    \label{eq:karsupbis}
    \frac{\sum_{j=k}^D \lambda_j^{p}}{k \lambda_k^{p}}& \le \frac{L-1}{1-L^{1-p/\nu_+}}&& \text{for $p>\nu_+$ and $1\le k\le D$.}
\end{align}
\end{lemma}

\begin{proof}
For inequality \eqref{eq:karinfpospbis}, write:
\[\sum_{j=k+1}^{Lk} \lambda_j^{p}    \ge  (L-1) k \lambda_{Lk}^p \ge L^{-p/\nu_-} (L-1) k  \lambda_k^p.
       \]
       We turn to \eqref{eq:karinfnegpbis}, for which we write
\begin{multline*}
    \sum_{j \le k} \lambda_j^{-p}
    \ge      \sum_{j=\lceil k/L \rceil}^{k} \lambda_j^{-p}
    \ge     \Big(k +1 - \Big\lceil \frac{k}{L} \Big\rceil \Big) \lambda_{\lceil k/L \rceil}^{-p} \ge (1-L^{-1})k L^{-p/\nu_-} \lambda_{k}^{-p}.
\end{multline*}
Finally, for \eqref{eq:karsupbis}, we have $\frac{\lambda_{Lk}}{\lambda_k} \le L^{-1/\nu_+}$ and
\begin{multline*}
    \sum_{j= k}^D \lambda_j^{p}
    = \sum_{\ell= 0}^{\lceil D/L \rceil} \sum_{i=kL^\ell}^{(kL^{\ell+1}-1)\wedge D} \lambda_i^p
     \le   \sum_{\ell = 0}^{\lceil D/L \rceil} k L^{\ell}(L-1) \lambda^p_{k L^{\ell}}\\
     \le  k  (L-1) \lambda_k^p \sum_{\ell \ge 0}  L^{\ell(1-p/\nu_+)}=\tfrac{(L-1)k\lambda_k^p}{1-L^{1-p/\nu_+}}.
 \end{multline*}
 \end{proof}

We will need the following auxiliary result:
\begin{lemma}\label{le:ineqlambda1}
    Suppose Assumptions {\bf (S)} and {\bf (R)} are satisfied
  with $2\rho>\nu_+$,
  then the  constant $t_\circ$ defined via \eqref{Eqt0}  is such that
  \begin{equation}
    \label{eq:zetafactor}
    t_\circ \ge \wt{\zeta} \lambda_{1}^{-1}, \qquad\text{ with }\qquad
    \wt{\zeta}:=\bigg( \frac{C_\circ \sqrt{D}(1-L^{1-2\rho/\nu_+})}{(L-1)\beta_+}
    \bigg)^{{1}/{2\rho}}.
  \end{equation}
  It follows, for any $k=1,\ldots,D$:
  \begin{equation}
    \label{eq:aux2}
    \delta^2 \lambda_k^{-2} \le C k^{2/\nu_-} V_{t_\circ}, \qquad
    \text{ with } C= \beta_-^{-2} L^{2/\nu_-} (1 \wedge \wt{\zeta})^{-2\rho}.
  \end{equation}
\end{lemma}
\begin{proof}
Using Assumption~\ref{AssGComp} and then \eqref{eq:karsupbis} gives
\[
  V_{t,\lambda} = \delta^2 \sum_{i=1}^D (\gamma_i^{(t)})^2
  \le \delta^2 \beta_+ t^{2\rho} \sum_{i=1}^D \lambda_i^{2\rho}
  \le \delta^2 \beta_+ t^{2\rho} \lambda_1^{2\rho} \frac{L-1}{1-L^{1-2\rho/\nu_+}}.
 \]
 By definition, $V_{t_\circ,\lambda}=C_\circ D^{1/2}\delta^2$ holds and \eqref{eq:zetafactor} follows.
 We then have as a consequence
 \[
 V_{t_\circ} \ge \delta^2 (\gamma_1^{(t_\circ)})^2 \lambda_1^{-2}
\ge \beta_-^2 \min(1,\wt{\zeta})^{2\rho}  \delta^2 \lambda_1^{-2} .
\]
Finally, for any $k=1,\ldots,D$, put $\ell := \lceil \log k/\log L \rceil$, then \eqref{Critbis} entails (notice $\lceil \lceil k/L^{i} \rceil/L \rceil
= \lceil k/L^{i+1} \rceil$ in the implicit iterations for the first inequality):
\[
\lambda^{-2}_{k} \le
L^{2\ell/\nu_-} \lambda_{\lceil k/L^{\ell} \rceil}^{-2}
= \lambda_1^{-2} L^{2\ell/\nu_-}
\le \lambda_1^{-2} (Lk)^{2/\nu_-}.
\]
Combining the two last displays yields \eqref{eq:aux2}.
\end{proof}

\begin{proof}[Proof of Proposition \ref{prop:suffcond}]
The monotonicity, continuity and limiting behaviour of $g(t,\lambda)$ in $t=0,\, t\rightarrow \infty$ for fixed $\lambda$
required from \ref{AssGMon} ensure the basic requirements on the filter sequence
(namely $t\mapsto\gamma_i^{(t)}$ continuous, $\gamma_i^{(0)}=0$ and $\gamma_i^{(t)}\uparrow 1$ as $t\to\infty$ ).
Since the spectral sequence $(\lambda_i)_{1\leq i \leq D}$ is nonincreasing, the monotonicity in $\lambda$ of $g(t,\lambda)$
ensures the validity of \ref{AssMon}. Similarly, Assumption \ref{AssGMonRatio} transparently ensures \ref{AssBias2}.

We turn to checking \ref{AssLambdaL2}. For this we use \eqref{eq:karinfnegpbis} with $p=2$, yielding
\[
    \forall k \ge 1\,:\,
    \frac{\sum_{j\le k} \lambda_j^{-2}}{k\lambda_{k}^{-2}}\ge    L^{-2/\nu_-}(1-L^{-1}) =: c_\lambda .
\]
We now check \ref{Assl1l2} for $t\ge t_\circ$.
Denote $\zeta:=\min(\wt{\zeta},1)$, where $\wt{\zeta}$ is from \eqref{eq:zetafactor}.
Introduce $j_t^* := \min\set{k: 1 \leq k \leq D, \lambda_k < \zeta/t}\cup \set{D+1}$.
Property $t \ge t_\circ$ and \eqref{eq:zetafactor} imply $j_t^*\ge 2$.
Assumption \ref{AssGComp} and \eqref{eq:specreg} ensure $\gamma_j^{(t)} \in [\zeta^\rho\beta_-,1]$
for all $j< j_t^*$, so that
\begin{equation}
  \label{eq:decl1l2}
  \sum_{i=1}^D \gamma_i^{(t)}  = \sum_{i=1}^{j^*_t-1} \gamma_i^{(t)} + \sum_{i=j^*_t}^D \gamma_i^{(t)} \le \zeta^{-\rho}\beta_-^{-1} \sum_{i=1}^{j^*_t-1} \big(\gamma_i^{(t)}\big)^{2} + \beta_+ t^\rho \sum_{i=j^*_t}^D \lambda_i^\rho.
\end{equation}
To control the second term (present only if $j^*_t \leq D$),
we note that, since $\rho > \nu_+$, \eqref{eq:karsupbis} yields
\[
\sum_{j\ge k} \lambda_j^\rho \le C k\lambda_k^\rho, \qquad \text{ where } C:= \frac{L-1}{1-L^{1-p/\nu_+}}.
\]
We apply this relation to $k=j^\ast_t$.  We deduce from $\lambda_{j_t^\ast} < \zeta/t$ and
$j_t^\ast \ge 2$
\[
t^\rho \sum_{i=j^*_t}^D \lambda_i^\rho \le C j^*_t t^\rho \lambda_{j^*_t}^\rho
\le 2C \zeta^\rho (j^*_t -1) \le 2C \zeta^{-\rho} \beta_-^{-2} \sum_{i=1}^{j^*_t-1}  (\gamma_i^{(t)})^2.
\]
Plugging this into \eqref{eq:decl1l2} yields \ref{Assl1l2} with $C_{\ell_1,\ell_2}= \zeta^{-\rho}(\beta_-^{-1} + 2C\beta_+ \beta_-^{-2})$.

We finally turn to \ref{AssVar}.
Without loss of generality we can assume $\beta_+ \ge 1$ in \eqref{AssGComp}.
As for all $A,B>0$
\begin{equation}
  \frac{1}{2}(A+B)^{-1} \le \min(A^{-1},B^{-1}) \le (A+B)^{-1} ,
    \label{eq:minineq}
\end{equation}
it follows that condition  \eqref{AssGComp} implies for all $t\ge t_\circ$:
\begin{equation}
\label{eq:condupperlower3}
\beta_-(1+(t\lambda)^{-\rho})^{-1} \le g(t,\lambda) \le 2(1+\beta_+^{-1} (t\lambda)^{-\rho})^{-1} \le 2 \beta_+
(1+(t\lambda)^{-\rho})^{-1}.
\end{equation}
Denote $h(t):=(1+t^{-\rho})^{-1}$; the above implies together with \eqref{eq:specreg}
that $\frac{\gamma_i^{(t)}}{h(t\lambda_i)} \in [\beta_-,2\beta_+]$. We infer that for $t\ge t' \ge t_\circ $:
\[
  \frac{V_{t}}{V_{t'}} = \frac{\sum_{i=1}^D (\gamma_{i}^{(t)} \lambda_i^{-1})^2}{\sum_{i=1}^D (\gamma_i^{(t')} \lambda_i^{-1})^2} \le \frac{4\beta_+^2}{\beta_-^2} \frac{\sum_{i=1}^D \paren{\lambda_i^{-1} h(t\lambda_i)}^2}{\sum_{i=1}^D \paren{\lambda_i^{-1} h(t'\lambda_i)}^2}
  =: \frac{4\beta_+^2}{\beta_-^2} \frac{H(t)}{H(t')},
  \]
while
  \[
  \frac{V_{t,\lambda}}{V_{t',\lambda}} = \frac{\sum_{i=1}^D (\gamma_{i}^{(t)})^2}{\sum_{i=1}^D (\gamma_i^{(t')})^2} \ge \frac{\beta_-^2}{4\beta_+^2} \frac{\sum_{i=1}^D h(t\lambda_i)^2}{\sum_{i=1}^D  h(t'\lambda_i)^2}
  =: \frac{\beta_-^2}{4\beta_+^2} \frac{G(t)}{G(t')}.
  \]
The desired bound \ref{AssVar} with $C_{V,\lambda} :=
(4\beta_+^2/\beta_-^2)^{1+\pi}$ for some constant $\pi>0$ follows from $ H(t)/H(t') \le (G(t)/G(t'))^{\pi}$.
It therefore suffices to establish for all $t\ge t_\circ$:
\begin{equation}
  \label{eq:diff}
  \frac{d}{dt}\log H(t) = \frac{H'(t)}{H(t)} \le \pi \frac{G'(t)}{G(t)}
  = \pi \frac{d}{dt}\log G(t),
\end{equation}
{\it i.e.}, to check the inequality
\begin{multline*}
    \frac{\sum_{i=1}^D \lambda_i^{-1} hh'(t\lambda_i)}{\sum_{i =1}^D\lambda_i^{-2} h(t\lambda_i)^2} = \rho t^{-\rho-1}
  \frac{\sum_{i=1}^D \lambda_i^{-\rho-2} \paren{1+\paren{t\lambda_i}^{-\rho}}^{-3}}
       {\sum_{i=1}^D \lambda_i^{-2} \paren{1+\paren{t\lambda_i}^{-\rho}}^{-2}}\\
       \le \pi t^{-\rho-1}
       \frac{\sum_{i=1}^D \lambda_i^{-\rho} \paren{1+\paren{t\lambda_i}^{-\rho}}^{-3}}
            {\sum_{i=1}^D \paren{1+\paren{t\lambda_i}^{-\rho}}^{-2}}
            = \pi \frac{\sum_{i=1}^D \lambda_i hh'(t\lambda_i)}{\sum_{i=1}^D h(t\lambda_i)^2}.
  \end{multline*}
Using \eqref{eq:minineq} again, it is sufficient to check the above inequality when replacing
everywhere $\paren{1+\paren{t\lambda_i}^{-\rho}}^{-1}$ by $\min(1,(t\lambda_i)^\rho)$,
  and $\pi$ by $\pi/32$.

Denoting $k^\ast_t:=\inf \set{k: 1\leq k \leq  : \lambda_k < t^{-1} }\cup \set{D+1}$,
  we thus have to establish the sufficient condition (for some constant $\pi$)
  \begin{equation}
  \frac{\sum_{i< k^*_t} \lambda_i^{-\rho-2} + \sum_{i = k^*_t}^D t^{3\rho} \lambda_i^{2\rho-2}}
       {
         \sum_{i< k^*_t} \lambda_i^{-2}+ \sum_{i = k^*_t}^D  t^{2\rho}  \lambda_i^{2\rho-2}}
       \le \frac{\pi}{32}
       \frac{
         \sum_{i < k^*_t} \lambda_i^{-\rho}
         + \sum_{i = k^*_t}^D t^{3\rho} \lambda_i^{2\rho}}
            {(k^*_t -1) + \sum_{i = k^*_t}^D  t^{2\rho}  \lambda_i^{2\rho}}.
            \label{eq:tylike}
  \end{equation}
Writing the left fraction as $(A_1+A_2)/(B_1+B_2)$ and the right fraction (without $\pi/32$) as $(A_3+A_4)/(B_3+B_4)$,
we check this relation by bounding $A_iB_j\le \frac{\pi}{32}
A_jB_i$ for $i=1,2$, $j=3,4$.
Without loss of generality we assume $1< k^*_t \le D$ (otherwise some products are just zero).
Let us  recall that \eqref{eq:karinfnegpbis} implies that
\begin{equation}
\label{eq:specialkar}
\forall k, 1 \leq k \leq D:\qquad
   k \lambda_{k}^{-\rho} \le  C \sum_{j\le k} \lambda_j^{-\rho},
  \end{equation}
for $C:= L^{\rho/\nu_-}/(1-L^{-1})$. We will also need below a similar
bound with $\lambda_{k}$ replaced by $\lambda_{k+1}$ on the left-hand side. For this,
notice that by Assumption {\bf (S)} we have $\frac{\lambda_k}{\lambda_{k+1}} \le \frac{\lambda_k}{\lambda_{Lk}} \le L^{1/\nu_-}$,
combining with \eqref{eq:specialkar} we get:
\begin{equation}
\label{eq:specialkar'}
\forall k, 1 \le k \le D:\qquad
k \lambda_{k+1}^{-\rho} \le k L^{\rho/\nu_-} \lambda_{k}^{-\rho} \le
C' \sum_{j\le k} \lambda_j^{-\rho}
  \end{equation}
with $C':= L^{2\rho/\nu_-}/(1-L^{-1})$.
The first term to handle is now (using \eqref{eq:specialkar}):
  \begin{multline*}
    A_1B_3=(k^*_t-1) \sum_{i<k^*_t } \lambda_i^{-\rho-2}  \le (k^*_t-1)
    \lambda_{k^*_t-1}^{-\rho} \sum_{i< k^*_t}
    \lambda_i^{-2}\\
    \le 
    C\big(
    {\sum_{i< k^*_t} \lambda_i^{-2}} \big) \big(
    {\sum_{i<k^*_t} \lambda_i^{-\rho}} \big)=CB_1A_3.
  \end{multline*}
  For the second term we clearly have $A_2B_4=A_4B_2$.
  The third term (using \eqref{eq:specialkar'} and the definition of $k^*_t$) is bounded as:
  \begin{align*}
    B_3A_2=(k^*_t-1) \sum_{i=k^*_t}^D t^{3\rho} \lambda_i^{2\rho-2} & \le
    C'  t^{\rho} \lambda_{k^*_t}^{\rho}
      (\sum_{i<k^*_t} \lambda_i^{-\rho})
      (\sum_{i=k^*_t}^D t^{2\rho} \lambda_i^{2\rho-2})\\
      & \le C'
      ( 
      \sum_{i< k^*_t} \lambda_i^{-\rho})  (\sum_{i=k^*_t}^D t^{2\rho} \lambda_i^{2\rho-2})=C'A_3B_2.
  \end{align*}
  For the fourth term we bound, using the definition of $k^*_t$,
    \[
     A_1B_4 =
      \sum_{i<k^*_t} \lambda_i^{-2}(t\lambda_i)^{-\rho} \sum_{i=k^*_t}^D t^{3\rho}\lambda_i^{2\rho}
      \le
      \sum_{i=k^*_t}^D t^{3\rho}\lambda_i^{2\rho} (
      \sum_{i< k^*_t} \lambda_i^{-2})=A_4B_1.
    \]
Hence, \eqref{eq:tylike} is established if we choose $\pi \ge 32 C'$.
\end{proof}

\subsection{Proof of Lemma \ref{LemVart}}\label{proof:LemVart}

We start with considering the case $t < \lambda_D^{-1}$.
Let us introduce the spectral distribution function $F(u)=\#\{i:\lambda_i\ge u\}$ for $u>0$. Then Assumption {\bf (S)} gives  for $k \leq \lfloor D/L\rfloor$:
\[ F(L^{1/\nu_-}u) \ge k \Leftrightarrow \lambda_k\ge L^{1/\nu_-}u\Rightarrow \lambda_{Lk}\ge u \Leftrightarrow F(u) \ge Lk,
\]
so that taking $k=F(L^{1/\nu_-} u)$ in the
above display yields
$F(u)\le \frac{L}{L-1} (F(u)-F(L^{1/\nu_-}u))$,
provided $F(L^{1/\nu_-} u) \leq \lfloor D/L\rfloor$.
We will apply the relation for $u=1/t$ below, and check
\[
u=t^{-1} > \lambda_D \geq L^{-1/\nu_-} \lambda_{\lceil \frac{D}{L} \rceil}
\geq L^{-1/\nu_-} \lambda_{\lfloor \frac{D}{L} \rfloor +1}
\Rightarrow F(L^{1/\nu_-} u) \leq \lfloor D/L\rfloor.
\]
Set again $k^*_t:=\inf\{k\ge 1:\lambda_k<t^{-1}\}\wedge(D+1)= F(1/t) +1$.
Under Assumption {\bf (R)} we conclude for $t\ge t_\circ$
\begin{align*}
\delta^{-2}V_{t,\lambda} = \sum_{i=1}^D (\gamma_i^{(t)})^2
&\le F(1/t)+\sum_{i=k^*_t}^D(t\lambda_i)^{-2}(\gamma_i^{(t)})^2\\
&\le \frac{L}{L-1} \sum_{i:1/t\le \lambda_i<L^{1/\nu_-}/t} (L^{-1/\nu_-}t\lambda_i)^{-2}+\sum_{i=k^*_t}^D(t\lambda_i)^{-2}(\gamma_i^{(t)})^2\\
&\le \frac{L^{1+2/\nu_-}}{L-1}t^{-2}\sum_{i=1}^D \lambda_i^{-2}(\gamma_i^{(t)}/\beta_-)^2\\
&= \tfrac{L^{1+2/\nu_-}}{L-1} \beta_-^{-2}t^{-2} \delta^{-2}V_t.
\end{align*}
This establishes the first inequality.
In the other direction,
denote, as in the proof of Prop. \ref{prop:suffcond}, $j^*_t:=\inf\{i: 1 \leq i
\leq D,\lambda_i<\zeta/t\}\cup \set{ D+1}$, where  $\zeta:=\min(\wt{\zeta},1)$ with
$\wt{\zeta}$  from \eqref{eq:zetafactor}. Under Assumption \ref{AssGComp} we have
\[
t^{-2}\delta^{-2}V_{t} = t^{-2} \sum_{i=1}^D \lambda_i^{-2}(\gamma_i^{(t)})^2
\le  \zeta^{-2} \sum_{i<j^*_t} (\gamma_i^{(t)})^2+ \beta_+^2 t^{2\rho-2}
\sum_{i=j^*_t}^D\lambda_i^{2\rho-2}.
\]
We  concentrate on the second term and assume without loss of generality
that $j^*_t \leq D$ (otherwise this term vanishes). If $D\ge Lj^*_t$, we use
the Karamata relations \eqref{eq:karsupbis}, then \eqref{eq:karinfpospbis},
and $2\rho-2>\nu_+$, to bound
\[
\sum_{i= j^*_t}^D \lambda_i^{2\rho-2}
\le  \frac{(L-1) j^*_t \lambda_{j^*_t}^{2\rho-2}}{1-L^{1-(2\rho-2)/\nu_+}}
\le \frac{L^{2\rho/\nu_-}}{(1-L^{1-(2\rho-2)/\nu_+})}
 \lambda_{j^*_t}^{-2} \sum_{i= j^*_t}^D \lambda_i^{2\rho}.
\]
If $D< Lj^*_t$, then directly using \eqref{Critbis} and $\rho \geq 1$:
\[
\sum_{i= j^*_t}^D \lambda_i^{2\rho-2} \le \lambda_{D}^{-2} \sum_{i= j^*_t}^D \lambda_i^{2\rho} \le L^{2\rho/\nu_-} \lambda_{j^*_t}^{-2} \sum_{i= j^*_t}^D \lambda_i^{2\rho} ,
\]
which implies that the inequality derived in the first case still holds.
Additionally, since $j^*_t \ge 2$ from Lemma~\ref{le:ineqlambda1} and
$t \geq t_\circ$, we have $\lambda_{j^*_t}\ge\lambda_{(L(j^*_t-1))\wedge D}\ge L^{-1/\nu_-}\lambda_{j^*_t-1}\ge L^{-1/\nu_-}\zeta t^{-1}$,
so that
\[
t^{2\rho-2}\sum_{i= j^*_t}^D \lambda_i^{2\rho-2} \le
\tfrac{\zeta^{-2}L^{2(\rho+1)/\nu_-}}{(1-L^{1-(2\rho-2)/\nu_+})}
\sum_{i= j^*_t}^D (t\lambda_i)^{2\rho}
\le
\tfrac{\zeta^{-2}L^{2(\rho+1)/\nu_-}}{(1-L^{1-(2\rho-2)/\nu_+})\beta_-^2}
\sum_{i= j^*_t}^D (\gamma_i^{(t)})^{2}.
\]
Altogether, we obtain the second inequality:
\[ t^{-2}V_t\le 
\tfrac{\zeta^{-2}L^{2(\rho+1)/\nu_-}}{(1-L^{1-(2\rho-2)/\nu_+})\beta_-^2}
V_{t,\lambda}.\]
We turn to the case $t\geq \lambda_D^{-1}$. Then Assumption \ref{AssGComp} implies
that for all $i=1,\ldots,D$ we have $ \gamma_i^{(t)} \geq \beta_-$.
It follows from this and \eqref{eq:karinfnegpbis}:
\begin{align*}
  \delta^{-2} V_t = \sum_{i=1}^D \lambda_i^{-2} (\gamma_i^{(t)})^2
  & \geq \beta_-^2 \sum_{i=1}^D \lambda_i^{-2}\\
  & \geq \beta_-^2 D \lambda_D^{-2} L^{-2/\nu_-}(1-L^{-1})\\
  & \geq \beta_-^2 L^{-2/\nu_-}(1-L^{-1}) \lambda_D^{-2}  \delta^{-2} V_{t,\lambda}\,,
\end{align*}
establishing the first bound in this case. The inequality in the other direction
follows directly from
\[
  \delta^{-2} V_t \leq D \lambda_D^{-2} \leq \beta_-^{-2} \lambda_D^{-2} \delta^{-2} V_{t,\lambda}.
  \]

\subsection{Maximal inequality for weighted $\chi^2$-variables with drift} \label{weighted chi2}

We first recall a result
on the concentration of weighted chi-squared type random variables.

\begin{lemma}[Laurent and Massart, Lemma 1 in \cite{LM}] \label{LemmaLaurentMassart}
Let $(Y_1,\ldots, Y_D)$ be i.i.d. $\mtc{N}(0,1)$ variables.
For nonnegative numbers $a_1,\ldots, a_D$ set
$$Z=\max_{1 \le k \le D}\sum_{i=1}^k a_i (Y_i^2-1).$$
With  $\|a\|_\infty:=\max_{1 \le i \le D}a_i$  the following inequalities hold for any $x>0$:
\begin{align}
&P\big(Z > 2 \norm{a} \sqrt{x}+ 2 \norm{a}_\infty x) < e^{-x},\label{EqLM1}\\
&P\big(Z < - 2 \norm{a} \sqrt{x}\big) < e^{-x}\label{EqLM2}
\end{align}
and also
\begin{align*}
P\big(Z > x)& < \exp\Big(-\frac{1}{4}\frac{x^2}{\norm{a}^2+\norm{a}_\infty x}\Big),\quad
P\big(Z <  - x\big)  < \exp\Big(-\frac{1}{4}\frac{x^2}{\norm{a}^2}\Big).
\end{align*}
\end{lemma}

Lemma \ref{LemmaLaurentMassart} is stated in a slightly more general setting, since the original result of Laurent and Massart \cite{LM}, based itself on  Lemma 8 in Birg\'e and Massart \cite{BM2}, has no maximum in $k$ for the definition of $Z$. The proof, however, is based on the classical Chernov bound argument, which readily carries over with a maximum: indeed, for $t\ge 0$ and $\lambda>0$,
$$P(Z \ge t) = P\Big(\max_{1 \le k \le D}e^{\lambda\sum_{i=1}^k a_i(Y_i^2-1)} \ge e^{\lambda t}\Big) \le e^{-\lambda t}\E\big[e^{\lambda \sum_{i=1}^D a_i (Y_i^2-1)}\big]$$
by Doob's maximal inequality applied to the submartingale $(e^{\lambda\sum_{i=1}^k a_i(Y_i^2-1)})_{1 \le k \le D}$.

\begin{lemma}\label{LemStochBounddelta}
Work under \ref{AssLambdaL2} of Proposition \ref{prop:suffcond}. Then, for every $\omega>0$ and every $x>0$ we have with probability at least $1-C_1e^{-C_2x}$, where $C_1,C_2>0$ only depend on $c_\lambda$ and $\omega$:
\[\max_{k\ge 1}\sum_{i=1}^k \lambda_i^{-2}(\eps_i^{2}-1-\omega)\le x\lambda_{\floor x\wedge D}^{-2}.
\]
\end{lemma}

\begin{proof}
From Lemma \ref{LemmaLaurentMassart} with $\norm{a}_{\ell^p(k)}^p=\sum_{i=1}^k\abs{a_i}^p$ for $p\ge 1$ and the usual modification for $p=\infty$ as well as  $\norm{\lambda^{-2}}_{\ell^p(k)}\le \norm{\lambda^{-2}}_{\ell^\infty(k)}k^{1/p}$ for $p=1,2$ and $\|\lambda^{-2}\|_{\ell^1(k)} \ge c_\lambda^2 k\lambda_k^{-2}=c_\lambda^2 k\|\lambda^{-2}\|_{\ell^\infty(k)}$ by \ref{AssLambdaL2}, we obtain for any integer $k \ge 1$
\begin{align*}
&P\Big(\max_{k=1,\ldots,D}\sum_{i=1}^k \lambda_i^{-2}(\eps_i^{2}-1-\omega)>r\Big)\\
&\le \sum_{k=1}^D P\Big(\sum_{i=1}^k \lambda_i^{-2}(\eps_i^{2}-1-\omega)>r\Big)\\
&= \sum_{k=1}^D P\Big(\sum_{i=1}^k \lambda_i^{-2}(\eps_i^{2}-1)>r+\omega\norm{\lambda^{-2}}_{\ell^1(k)}\Big)\\
&\le \sum_{k=1}^D P\Big(\sum_{i=1}^k \lambda_i^{-2}(\eps_i^{2}-1)>r+\omega c_\lambda^2 k\norm{\lambda^{-2}}_{\ell^\infty(k)}\Big)\\
&\le \sum_{k=1}^D \exp\Big(-\frac 14 \frac{(r+\omega c_\lambda^2 k\norm{\lambda^{-2}}_{\ell^\infty(k)})^2} { \norm{\lambda^{-2}}_{\ell^2(k)}^2+\norm{\lambda^{-2}}_{\ell^\infty(k)}(r+\omega c_\lambda^2 k\norm{\lambda^{-2}}_{\ell^\infty(k)})}\Big)\\
&\le \sum_{k=1}^D \exp\Big(-\frac 14 \frac{(r+\omega c_\lambda^2 k\norm{\lambda^{-2}}_{\ell^\infty(k)})^2} { k\norm{\lambda^{-2}}_{\ell^\infty(k)}^2+\norm{\lambda^{-2}}_{\ell^\infty(k)}(r+\omega c_\lambda^2 k\norm{\lambda^{-2}}_{\ell^\infty(k)})}\Big)\\
&= \sum_{k=1}^D \exp\Big(-\frac 14 \frac{(r\lambda_k^2+c_\lambda^2\omega k)^2} { (1+c_\lambda^2\omega) k+r\lambda_k^2}\Big)\\
&\le \sum_{k=1}^D \exp\Big(-\frac {c_\lambda^2\omega}{4(1+c_\lambda^2\omega)} (r\lambda_k^2+c_\lambda^2\omega k)\Big)\\
&\le \sum_{k=1}^{k^\ast} \exp\Big(-\frac {c_\lambda^2\omega}{4(1+c_\lambda^2\omega)} r\lambda_k^2\Big)+\sum_{k=k^\ast+1}^D\exp\Big(-\frac {c_\lambda^2\omega}{4(1+c_\lambda^2\omega)}c_\lambda^2\omega k\Big)\\
&\le k^\ast \exp\Big(-\frac {c_\lambda^2\omega  r\lambda_{k^\ast}^2}{4(1+c_\lambda^2\omega)}\Big)+\frac{1}{e^{c_\lambda^4\omega^2/(4+4c_\lambda^2\omega)}-1} \exp\Big(-\frac {c_\lambda^4\omega^2 k^\ast}{4(1+c_\lambda^2\omega)}\Big)
\end{align*}
for any $k^\ast$. The choice $k^\ast=\floor x\wedge D$ and $r=x\lambda_{k^\ast}^{-2}$ yields the asserted deviation bound with suitable constants $C_1,C_2>0$.
\end{proof}

\bibliographystyle{plain}       
\bibliography{biblioBHR}           

\begin{thebibliography}{10}

\bibitem{B}
A.B. Bakushinskii.
\newblock Remarks on the choice of regularization parameter from
  quasioptimality and relation tests.
\newblock {\em Zh. Vychisl. Mat. i Mat. Fiz. (Russian)}, 8:1258--1259, 1984.

\bibitem{BR2}
F.~Bauer and M.~Rei{\ss}.
\newblock Regularization independent of the noise level: an analysis of
  quasi-optimality.
\newblock {\em Inverse Problems}, 24(5):055009, 2008.

\bibitem{BGT}
N.~H. Bingham, C.~M. Goldie, and J.~L. Teugels.
\newblock {\em Regular Variation.}
\newblock Cambridge University Press, 1989.

\bibitem{BM2}
L.~Birg\'e and P.~Massart.
\newblock Minimum contrast estimators on sieves: exponential bounds and rates
  of convergence.
\newblock {\em Bernoulli}, 4(3):329--375, 1998.

\bibitem{BHMR}
N.~Bissantz, T.~Hohage, A.~Munk, and F.~Ruymgaart.
\newblock Convergence rates of general regularization methods for statistical
  inverse problems and applications.
\newblock {\em SIAM Journal on Numerical Analysis}, 45:2610--2636, 2007.

\bibitem{ourpapercutoff}
G.~Blanchard, M.~Hoffmann, and M.~Rei\ss.
\newblock Early stopping for statistical inverse problems via truncated {SVD}
  estimation.
\newblock Technical report, arXiv: 1710.07278, 2017.

\bibitem{BM}
G.~Blanchard and P.~Math\'e.
\newblock Discrepancy principle for statistical inverse problems with
  application to conjugate gradient iteration.
\newblock {\em Inverse Problems}, 28:pp. 115011, 2012.

\bibitem{BL}
L.~Brown and M.~Low.
\newblock Asymptotic equivalence of nonparametric regression and white noise.
\newblock {\em Annals of Statistics}, 24:2384--2398, 1996.

\bibitem{C}
L.~Cavalier.
\newblock Inverse problems in statistics.
\newblock In {\em Inverse problems and high-dimensional estimation}, pages
  3--96. Lecture Notes in Statistics 203, Springer, 2011.

\bibitem{CG}
L.~Cavalier and Y.~Golubev.
\newblock Risk hull method and regularization by projections of ill-posed
  inverse problems.
\newblock {\em Annals of Statistics}, 34:1653--1677, 2006.

\bibitem{CHR}
A.~Cohen, M.~Hoffmann, and M.~Rei\ss.
\newblock Adaptive wavelet {Galerkin} methods for linear inverse problems.
\newblock {\em SIAM Journal on Numerical Analysis}, 42(4):1479--1501, 2004.

\bibitem{DNT}
D.~Djurci\'c, R.~Nikoli\'ca, and A.~Torgasev.
\newblock The weak asymptotic equivalence and the generalized inverse.
\newblock {\em Lithuanian Mathematical Journal}, 50(1):34--42, 2010.

\bibitem{EHN}
H.~Engl, M.~Hanke, and A.~Neubauer.
\newblock {\em Regularization of Inverse Problems.}
\newblock Kluwer Academic Publishers, London, 1996.

\bibitem{F}
A.~Fleige.
\newblock Characterizations of monotone $\mathcal{O}$-regularly varying
  functions by means of indefinite eigenvalue problems and help type
  inequalities.
\newblock {\em Journal of Mathematical Analysis and Applications},
  412:1479--1501, 2014.

\bibitem{G}
Y.~Golubev.
\newblock Adaptive spectral regularizations of high dimensional linear models.
\newblock {\em Electronic Journal of Statistics}, 5:1588--1617, 2011.

\bibitem{Ha}
P.~C. Hansen.
\newblock {\em Discrete inverse problems: insight and algorithms, Fundamentals
  of Algorithms 7}.
\newblock SIAM, 2010.

\bibitem{KS}
I.~Karatzas and S.E. Schreve.
\newblock {\em Brownian motion and stochastic calculus.}
\newblock Springer Berlin Heidelberg, 2nd edition, 1991.

\bibitem{Ki}
S.~Kindermann.
\newblock Discretization independent convergence rates for noise level-free
  parameter choice rules for the regularization of ill-conditioned problems.
\newblock {\em Electronic Transactions on Numerical Analysis}, 40:58--81, 2013.

\bibitem{LM}
B.~Laurent and P.~Massart.
\newblock Adaptive estimation of a quadratic functional by model selection.
\newblock {\em The Annals of Statistics}, 28:1302--1338, 2000.

\bibitem{LuM}
S.~Lu and P.~Math\'e.
\newblock Discrepancy based model selection in statistical inverse problems.
\newblock {\em Journal of Complexity}, 30:386--407, 2014.

\bibitem{MR}
B.~Mair and F.H. Ruymgaart.
\newblock Statistical estimation in {Hilbert} scale.
\newblock {\em SIAM Journal on Applied Mathematics}, 56:1424--1444, 1996.

\bibitem{MP}
P.~Math{\'e} and S.~V. Pereverzev.
\newblock Geometry of linear ill-posed problems in variable {Hilbert} scales.
\newblock {\em Inverse problems}, 19(3):789, 2003.

\bibitem{PS}
S.~Pereverzev and E.~Schock.
\newblock On the adaptive selection of the parameter in regularization of
  ill-posed problems.
\newblock {\em SIAM Journal on Numerical Analysis}, 43:2060--2076, 2005.

\bibitem{Re}
M.~{Rei{\ss}}.
\newblock {Asymptotic equivalence for nonparametric regression with
  multivariate and random design.}
\newblock {\em {Annals of Statistics}}, 36(4):1957--1982, 2008.

\bibitem{Wa}
G.~Wahba.
\newblock Practical approximate solutions to linear operator equations when the
  data are noisy.
\newblock {\em SIAM Journal on Numerical Analysis}, 14(4):651--667, 1977.

\end{thebibliography}

\checknbdrafts

\end{document}